\documentclass{amsart}[12pt]
\usepackage{graphicx,url,amsthm,xspace,enumitem,subfigure,amssymb, caption}
\usepackage[dvipsnames]{xcolor}
\usepackage{amscd, amsmath}
\usepackage{multirow, overpic}
\usepackage{tikz, float, amsthm}

\usepackage{xy}
\xyoption{all}

\usepackage{pinlabel}
\usepackage[linktocpage]{hyperref}

\hypersetup{
  colorlinks,
  citecolor=Purple,
  linkcolor=Black,
  urlcolor=Purple}
  
\usepackage[margin=2.7cm]{geometry}

\newcommand{\C}{\mathbb{C}}
\renewcommand{\P}{\mathbb{P}}
\newcommand{\R}{\mathbb{R}}

\newcommand{\Z}{\mathbb{Z}}
\newcommand{\cptwo}{\C\P^2}
\newcommand{\cpone}{\C\P^1}
\newcommand{\cptwobar}{\overline{\C\P}\,\!^2}

\newcommand{\std}{std}

\theoremstyle{plain}
\newtheorem*{theorem*}{Main Result}
\newtheorem{theorem}{Theorem}

\newtheorem{prop}[theorem]{Proposition}

\theoremstyle{definition}

\newtheorem{definition}[theorem]{Definition}

\theoremstyle{remark}
\newtheorem{remark}[theorem]{Remark}

\numberwithin{equation}{section}

\title{An introduction to Weinstein handlebodies for complements of smoothed toric divisors}

\author[Acu]{Bahar Acu}
\address[B.\ Acu]{Department of Mathematics \\ Northwestern University \\ Evanston \\ IL \\ U.S.A.}
\email{baharacu@northwestern.edu}

\author[Capovilla-Searle]{Orsola Capovilla-Searle}
\address[O.\ Capovilla-Searle]{Department of Mathematics \\ Duke University \\ Durham \\ NC \\ U.S.A.}
\email{ocapovilla@gmail.com}

\author[Gadbled]{Agn\`{e}s Gadbled}
\address[A.\ Gadbled]{Universit\'e Paris-Saclay \\ CNRS \\ Laboratoire de Math\'ematiques d'Orsay \\ 91405 \\Orsay \\ France}
\email{agnes.gadbled@universite-paris-saclay.fr}

\author[Marinkovi\'c]{Aleksandra Marinkovi\'c}
\address[A.\ Marinkovi\'c]{Matemati\v{c}ki Fakultet \\ Belgrade \\ Serbia}
\email{aleks@matf.bg.ac.rs}

\author[Murphy]{Emmy Murphy}
\address[E.\ Murphy]{Department of Mathematics \\ Northwestern University \\ Evanston \\ IL \\ U.S.A.}
\email{e\_murphy@math.northwestern.edu}

\author[Starkston]{Laura Starkston}
\address[L.\ Starkston]{Department of Mathematics \\ UC Davis\\ Davis \\ CA \\ U.S.A.}
\email{lstarkston@math.ucdavis.edu}

\author[Wu]{Angela Wu}
\address[A.\ Wu]{Department of Mathematics \\ University College London \\ London \\ United Kingdom}
\email{angela.wu.17@ucl.ac.uk}



\begin{document}

\begin{abstract}
In this article, we provide an introduction to an algorithm for constructing Weinstein handlebodies for complements of certain smoothed toric divisors using explicit coordinates and a simple example. This article also serves to welcome newcomers to Weinstein handlebody diagrams and Weinstein Kirby calculus. Finally, we include one complicated example at the end of the article to showcase the algorithm and the types of Weinstein Kirby diagrams it produces.
\end{abstract}
\maketitle

\section{Introduction}

A key way to study closed symplectic manifolds is to break them down into two more easily understood parts: a neighborhood of a \emph{divisor} and a complementary \emph{Weinstein domain}. A divisor is a symplectic submanifold of co-dimension $2$. One can allow this submanifold to have certain controlled singularities, such as normal crossing singularities or more general singularities modeled on complex hypersurfaces. Donaldson proved that every symplectic manifold has a divisor \cite{Donaldson} and Giroux proved that this divisor can be chosen such that the complement of a regular neighborhood admits the structure of a Weinstein domain \cite{GirouxICM,Giroux}. A \emph{Weinstein domain} is a symplectic manifold with convex contact type boundary which can be broken down into symplectic handles modeled and glued as described by Weinstein \cite{Weinstein}. The symplectic topology of a Weinstein manifold is encoded in the attaching spheres of the handles. The attaching spheres can be represented by isotropic and Legendrian knots in the front projection. This Weinstein handlebody diagram gives a combinatorial/diagrammatic method to encode a symplectic manifold. There is a calculus of moves which relates different diagrams for equivalent Weinstein manifolds \cite{Gompf,DingGeiges}.

Recently, there has been increased study of symplectic divisors in symplectic manifolds, particularly in the case when the complement is Weinstein. Some of the motivation comes from homological mirror symmetry, where generalizing the link between coherent sheaves and Fukaya categories to larger classes of manifolds has required one to consider a mirror pair that includes not only a space, but also a divisor \cite{AurouxTduality}. One way to associate a Fukaya category for a divisor pair, is to look at the wrapped Fukaya category of the complement of the divisor. The Weinstein handle decomposition is key to understanding the wrapped Fukaya category, due to recent results that the co-cores of the handles generate the category \cite{CDGG, GPS1, GPS2}. The Floer homology of these co-cores is intrinsically tied to the Legendrian DGA of the Weinstein handlebody diagram~\cite{BEE,Ekholm, EkholmLekili} which is combinatorially calculated by Ekholm-Ng \cite{EkholmNg}.\smallskip

An important class of symplectic manifolds are \emph{toric manifolds}. These have been studied extensively as they form a large class of examples of integrable systems because of the symmetry provided by the Hamiltonian action of a torus on such manifolds. According to the famous Delzant classification, all compact symplectic toric $2n$-manifolds are uniquely (up to equivariant symplectomorphism) determined by convex $n$-dimensional polytopes, which correspond to the orbit space of the action. Much of the symplectic information can be encoded in the combinatorics of these polytopes known as Delzant polytopes.
 Moreover, toric manifolds have their origin in algebraic geometry, and they come by definition with a fibration by tori given by the action, so that they have been among the first cases of interest for homological mirror symmetry, especially in view of the SYZ philosophy (see for instance~\cite{Abouzaid}).
Every compact toric symplectic manifold is naturally equipped with a toric divisor. This is precisely the set of all points with non trivial stabilizer and the fixed points of the toric action are normal crossing singularities of the divisor. This can also be understood in terms of the moment map image polytope: the toric divisor is the preimage of the faces of the polytope under the moment map. The complement of a neighborhood of the divisor is symplectomorphic to a Weinstein domain whose completion is $T^*T^n$. (The complement of the divisor is the preimage of the interior of the polytope under the moment map, which is $T^n\times P$ where $P$ is a convex open subset of $\R^n$.) Hypersurfaces with normal crossing singularities can naturally be deformed to become less singular at the expense of increasing the topological complexity of the divisor and its complement. A toric manifold, together with its toric divisor or any smoothing of the divisor is a log Calabi-Yau pair which is a convenient setting for studying mirror symmetry of a space with a divisor~\cite{GrossHackingKeel, HackingKeating}.\smallskip

A manifold of dimension $4$ will have symplectic surface divisors. Normal crossing singularities in this dimension are just positive transverse intersections of two smooth branches, or \emph{nodes}. A deformation of this node smooths out the surface, trading the node for an annular tube which thus joins two different components or increases the genus of the surface. For a toric $4$-manifold, the complement of the (fully singular) toric divisor looks like $T^*T^2$, which has a natural Weinstein structure described by a diagram discovered by Gompf \cite{Gompf}. 

The toric divisor is the preimage of the facets of the Delzant polytope $\Delta$, and the nodes are the preimages of the vertices of $\Delta$. Since there is a one to one correspondence between the nodes and vertices, we will use the same notation to denote both a node and its moment map image vertex. Each vertex $V\in \Delta$, has a corresponding ray $r$ based at $V$ defined as the sum of the primitive edge vectors of $\Delta$ adjacent to $V$ and pointing outward from $V$. 
\begin{definition}\label{def:centered}
A toric manifold with a chosen subset $\{V_1, \ldots, V_n\}$ of the nodes is $\{V_1, \ldots , V_n\}$-centered, if the corresponding rays $r_1, \ldots , r_n$ all intersect at a common single point in the interior of its Delzant polytope.
\end{definition}
We show in \cite[Theorem 4.1]{ACGMMSW} that if a toric $4$-manifold is $\{V_1, \ldots , V_n\}$-centered, then the complement of the toric divisor smoothed at the nodes $\{V_1, \ldots, V_n\}$ has a Weinstein structure which we can explicitly describe. 

In this article, we explain an algorithm to produce a Weinstein handlebody diagram for the complement of any divisor obtained by smoothing the $\{V_1, \ldots, V_n\}$-nodes of a toric divisor in a $\{V_1, \ldots, V_n\}$-centered toric $4$-manifold. We prove the Weinstein handlebody produced by this algorithm is Weinstein homotopic to the complement of the smoothed toric divisor in our later paper \cite{ACGMMSW}. That article also explains how toric moment data determines the input to our algorithm, provides a detailed exploration of the centered hypothesis, and includes many more examples.
This article focuses on the most accessible example, as well as showcasing one fun case.

\subsection*{Main Results} 
There exists an algorithm to produce a Weinstein handlebody diagram for the complement of a toric divisor smoothed at the subset of nodes $\{V_1, \ldots, V_n\}$ in a $\{V_1, \ldots, V_n\}$-centered toric $4$-manifold.
\begin{enumerate}
\item Applying this algorithm to $\C\P^2$ smoothed in one node yields the self-plumbing of $T^*S^2$ as illustrated in Figure \ref{algorithm2}. Moreover, the same output is obtained for the complement of a toric divisor in \textit{any} toric 4-manifold smoothed in one node.
\begin{figure}[!ht]
\begin{center}
\begin{tikzpicture}
\node[inner sep=0] at (0,0) {\includegraphics[width=8 cm]{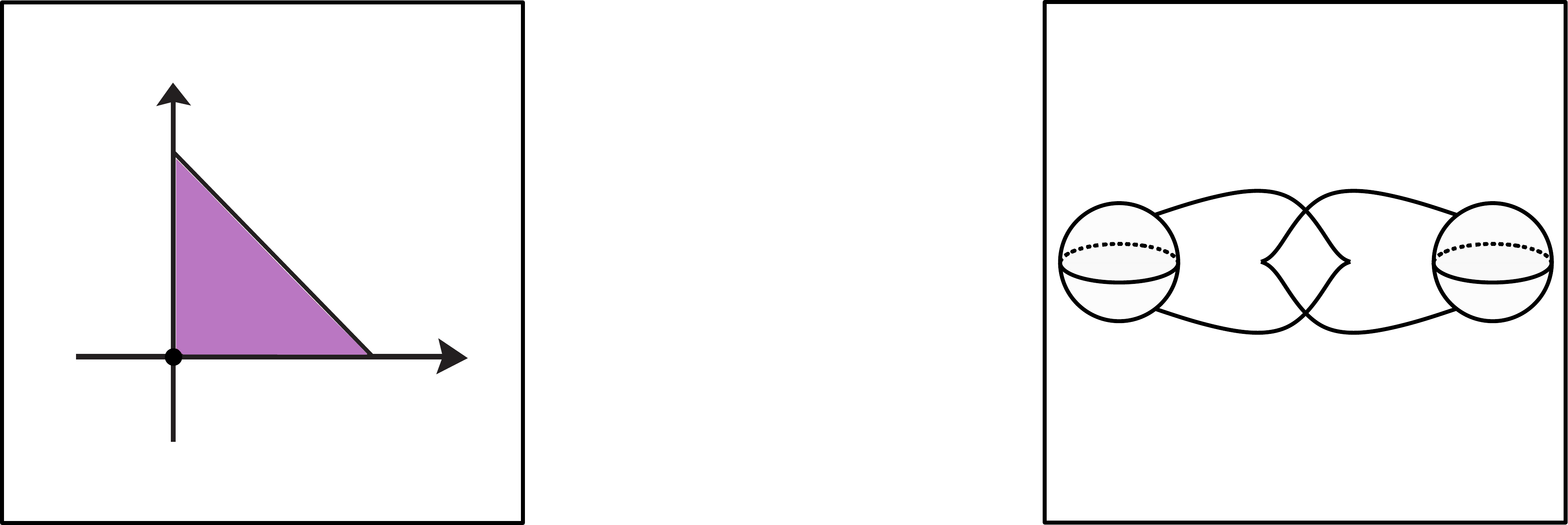}};
\draw[->, line width=0.8mm] (1,0) -- (1.3,0);
\draw[line width=0.8mm] (-1.2,0) -- (-1,0);
\draw (-1,-0.5) rectangle (1,0.5) node[pos=.5] {Algorithm};
\end{tikzpicture}
\caption{}
\label{algorithm2}
\end{center}
\end{figure}

\vspace{-.1in}

\item Applying the algorithm to a monotone $\C\P^2 \# 3\cptwobar$ smoothed at all six nodes yields a Weinstein manifold constructed by attaching $2$-handles to the $4$-ball along a $5$-component Legendrian link as in Figure~\ref{algorithm1}.
\begin{figure}[!ht]
\begin{center}
\begin{tikzpicture}
\node[inner sep=0] at (0,0) {\includegraphics[width=8 cm]{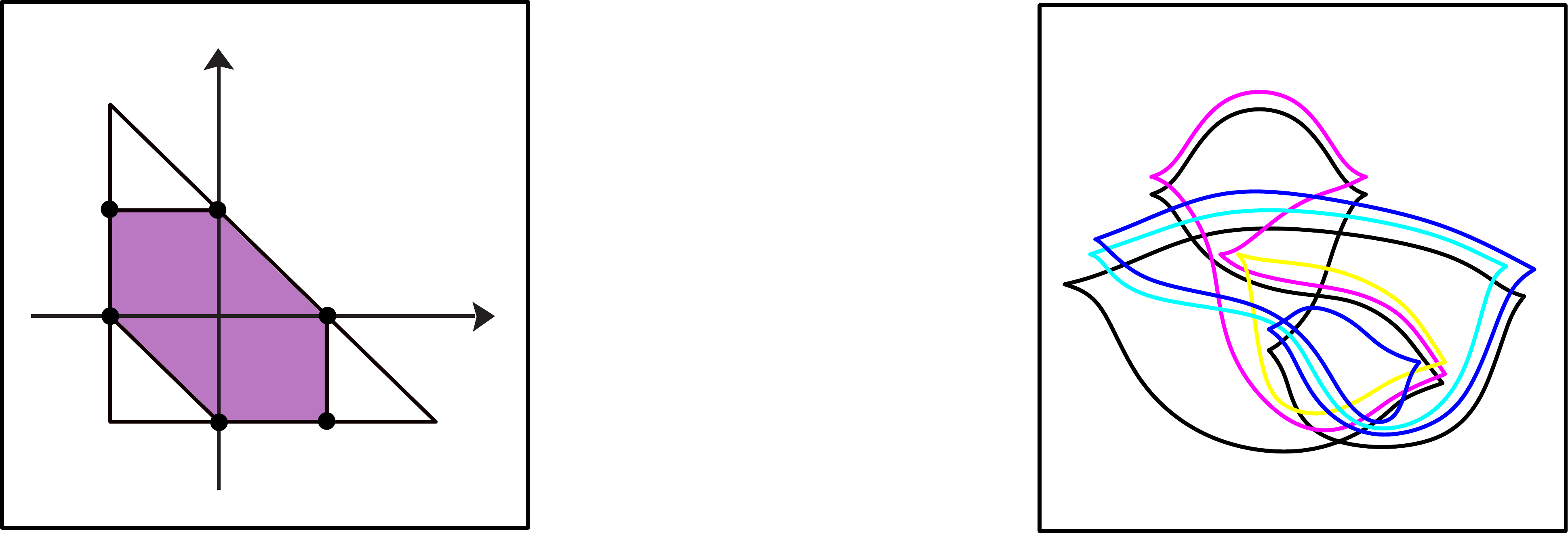}};
\draw[->, line width=0.8mm] (1,0) -- (1.3,0);
\draw[line width=0.8mm] (-1.2,0) -- (-1,0);
\draw (-1,-0.5) rectangle (1,0.5) node[pos=.5] {Algorithm};
\end{tikzpicture}
\caption{}
\label{algorithm1}
\end{center}
\end{figure}
\end{enumerate}

We would like to note that the handlebody diagram in the first example above has already been observed by Casals-Murphy in \cite{CM}, viewed as a Weinstein handlebody for the complement of an affine smooth conic in $\C^2$ (which is the same as the complement of a smooth conic together with a generic line in $\cptwo$, which is the once smoothed toric divisor in $\cptwo$). We obtain this diagram from a completely different method and provide a systematic recipe which applies much more generally. This first example provides an accessible way to explain the steps of our more general algorithm. \smallskip

This paper is organized as follows. In Section 2, we give definitions and discuss the relevant preliminary background on Weinstein Kirby calculus. In Section 3, we provide a picture into how to see the main structure of the handle attachment needed to obtain the complement of the smoothed divisor by describing the core and co-core of the handle in the smoothing local model. The remainder of the paper demonstrates the algorithm for producing the desired handlebody diagrams. In Section~4, we produce a Weinstein handlebody diagram for the complement of a toric divisor smoothed in one node and apply sequences of Kirby calculus moves to simplify the diagrams. In Section~5, we present a more complicated example, coming from $\cptwo \# 3\cptwobar$, with the toric divisor smoothed at all six nodes, to showcase the scope of applications and the corresponding Weinstein Kirby diagrams.

\subsection*{Acknowledgments} This project was initiated at the \href{https://icerm.brown.edu/topical_workshops/tw19-4-wiscon/}{2019 Research Collaboration Conference for Women in Symplectic and Contact Geometry and Topology (WiSCon)} that took place on July 22-26, 2019 at ICERM. The authors would like to extend their gratitude to the WiSCon organizers and the hosting institution ICERM and their staff for their hospitality. The authors would like to thank Lenhard Ng for useful conversations. BA and AM would like to thank the Oberwolfach Research Institute for Mathematics for hosting them during the completion of this project. AG was partially supported by Wallenberg grant no. KAW 2016-0440 and the Fondation Math\'ematique Jacques Hadamard. AM is partially supported by Ministry of Education and Science of Republic of Serbia, project ON174034.
LS is supported by NSF grant no. DMS 1904074. AW is supported by the Engineering and Physical Sciences Research Council [EP/L015234/1], the EPSRC Centre for Doctoral Training in Geometry and Number Theory (The London School of Geometry and Number Theory), University College London. OC-S is supported by NSF Graduate Research Fellowship under grant no. DGE-1644868.

\section{Weinstein handlebodies and Kirby calculus}


\subsection{Weinstein handle structure} 

A \emph{Liouville vector field} $Z$ for a symplectic manifold $(W,\omega)$ is a vector field satisfying $\mathcal{L}_Z \omega = \omega$. By Cartan's formula for the Lie derivative and the fact that the symplectic form is closed, this is equivalent to saying $d(\iota_Z\omega)=\omega$ (here $\iota$ denotes the interior product where $\iota_Z\omega(\cdot) = \omega(V,\cdot)$). In particular, when there exists a Liouville vector field, the symplectic structure is exact. The $1$-form $\lambda = \iota_Z\omega$ which satisfies $d\lambda=\omega$ is called the \emph{Liouville form}. For an introduction to these ideas, see \cite{mcduff2017introduction}. The primary use of Liouville vector fields is to glue symplectic manifolds along contact type boundaries. When the Liouville vector field is transverse to the boundary, it defines a contact structure on the boundary and can be used to identify a collared neighborhood of the boundary with a piece of the symplectization of that contact boundary.\smallskip

When Weinstein defined a model of a handle decomposition for symplectic manifolds \cite{Weinstein}, he equipped the handle with a Liouville vector field so that the gluing of the handle attachment could be performed using only contact information on the boundary. More specifically, the handle attachment is completely specified by a Legendrian attaching sphere, or an isotropic attaching sphere together with data on its normal bundle. The limitation is that the index of the handle is required to be less than or equal to $n$ in a $2n$ dimensional manifold. In particular, Weinstein $4$-manifolds must be built entirely from handles of index $0$, $1$, and $2$. \smallskip

The model Weinstein handle of index $k$ in dimension $2n$ for $k\leq n$ is a subset of $\R^{2n}$ with coordinates $(x_1,y_1,\cdots, x_n,y_n)$, with the standard symplectic structure $\omega = \sum_j dx_j\wedge dy_j$ and Liouville vector field
$$Z_k = \sum_{j=1}^k  \left( -x_j\partial_{x_j}+2y_j\partial_{y_j} \right) + \sum_{j={k+1}}^n  \left(\frac{1}{2}x_j\partial_{x_j}+\frac{1}{2} y_j\partial_{y_j}\right).$$

\begin{figure}
	\centering
	\includegraphics[scale=.5]{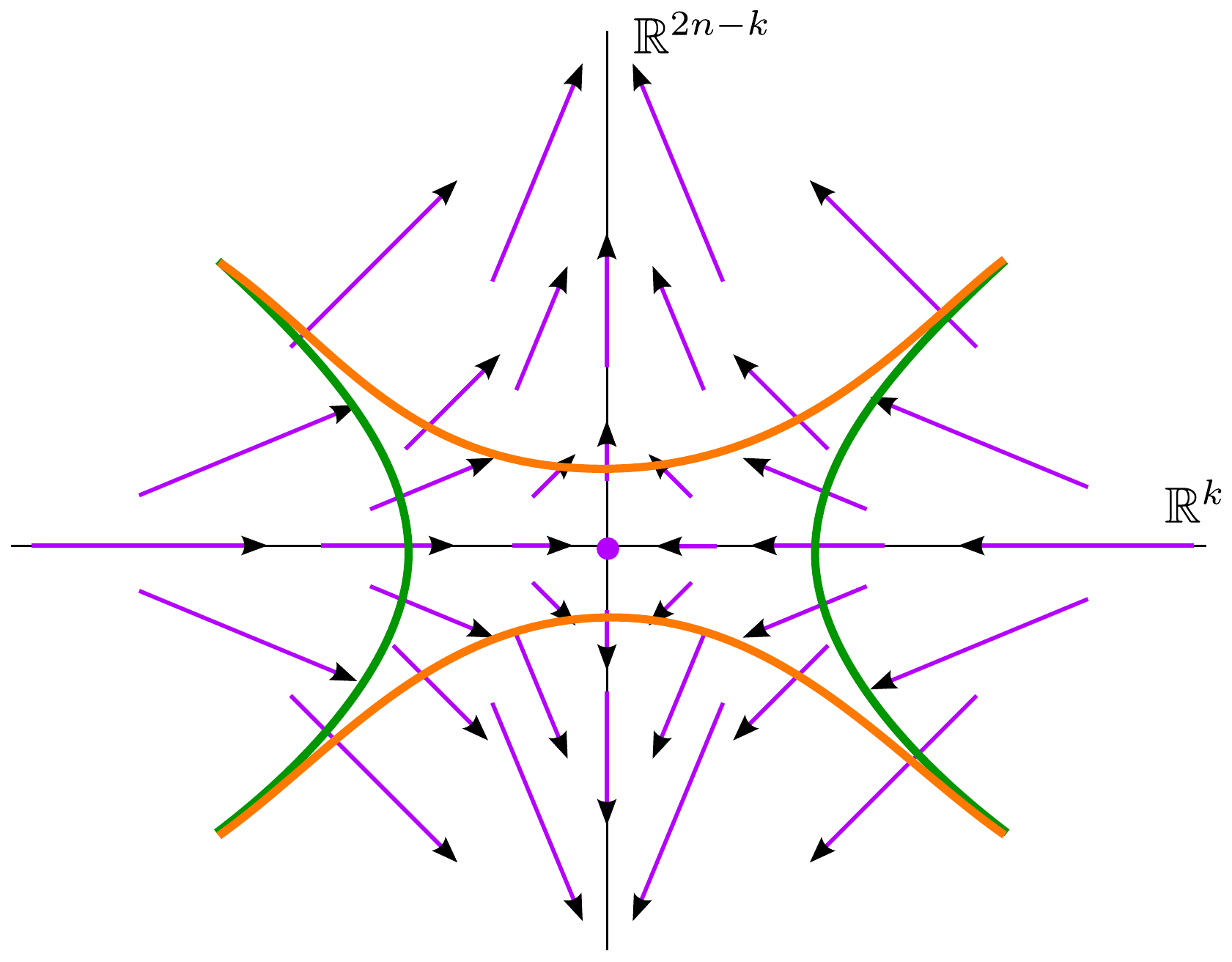}
	\caption{A sketch of the model for the Weinstein handle of index $k$, with the associated gradient-like Liouville vector field $Z_k$.}
	\label{fig:model}
\end{figure}

As with smooth handle theory, the handles are in one to one correspondence with the critical points of a Morse function. The Liouville vector field agrees with the gradient of such a Morse function (for some choice of metric), in other words, the Liouville vector field is \emph{gradient-like}. In the model index $k$ handle, the Liouville vector field is the gradient (with the standard Euclidean metric) of the function  
$$\phi_k = \sum_{j=1}^k \left(-\frac{1}{2} x_j^2 + y_j^2\right) + \sum_{j={k+1}}^n \left(  \frac{1}{4}x_j^2+\frac{1}{4}y_j^2 \right).$$

The handle can be considered to be the subset of $\R^{2n}$ given by $D^k\times D^{2n-k}$ where the first factor corresponds to the coordinates $(x_1,\cdots, x_k)$ and the second corresponds to the remaining coordinates $(x_{k+1},\cdots, x_n,y_1,\cdots, y_n)$. The key terminology for important parts of the handle is as follows.

\begin{itemize}
	\item The \emph{core} of the handle is $D^k\times \{0\}$ where $x_{k+1}=\cdots=x_n=y_1=\cdots=y_n=0$. This is the \emph{stable manifold} of flow-lines of $Z_k$ which limit positively towards the zero at the origin.
	\item The \emph{co-core} of the handle is $\{0\}\times D^{2n-k}$ where $x_1=\cdots=x_k=0$. This is the \emph{unstable manifold} of flow-lines of $Z_k$ which limit negatively towards the zero at the origin.
	\item The \emph{attaching sphere} is the boundary of the core, $S^{k-1}\times\{0\}$. This will be identified with an isotropic sphere in the boundary of the existing manifold to which the handle is attached.
	\item The \emph{attaching region} is a neighborhood of the attaching sphere $S^{k-1}\times D^{2n-k}$. This is the entire part of the handle which will be glued on to a piece of the boundary of the existing manifold when the handle is attached. Therefore the Liouville vector field $Z_k$ points inward into the handle along this part of the boundary (it is \emph{concave}).
	\item The \emph{belt sphere} is the boundary of the co-core $\{0\}\times S^{2n-k-1}$. It is a co-isotropic sphere in the boundary of the manifold obtained after attaching the handle.
\end{itemize}

In general we can piece together the Liouville vector fields on the handles, and put together adjusted versions of the locally defined Morse functions to get a global Morse function on the manifold. A Weinstein structure is often encoded analytically as a quadruple $(W,\omega, Z, \phi)$ where $W$ is a smooth manifold, $\omega$ is a symplectic structure on $W$, $Z$ is a Liouville vector field for $\omega$ on $W$, and $\phi$ is a Morse function such that $Z$ is the gradient-like for $\phi$.

\begin{remark}
	When a manifold with a Weinstein structure has contact type boundary, it is called a \emph{Weinstein domain}. Such a domain can be extended by a cylindrical end to make the Liouville vector field complete to give a non-compact infinite volume \emph{Weinstein manifold}.
\end{remark}


\subsection{Weinstein Kirby calculus} \label{s:kirby}

The data needed to encode the Weinstein domain are the attaching maps. In dimension $4$, the attaching map of a handle is completely determined by the Legendrian or isotropic attaching sphere. The attaching sphere of a $1$-handle is a pair of points. Diagramatically we draw a pair of $3$-balls implicitly identified by a reflection, representing the attaching region $S^0\times D^3$. The attaching sphere of a $2$-handle is a Legendrian embedded circle (a Legendrian knot). The $4$-dimensional $2$-handle attachment is determined by the knot together with a framing, but in the Weinstein case, the framing is determined by the contact structure. More specifically, the contact planes along a Legendrian knot determine a framing by taking a vector field transverse to the contact planes. The contactomorphism gluing the attaching region of the $2$-handle to the neighborhood of the Legendrian identifies the product framing in the $2$-handle with the $tb-1$ framing. Here $tb$ denotes the contact framing (Thurston-Bennequin number), which is identified with an integer by looking at the difference between the contact framing and the Seifert framing (this must be appropriately interpreted when the diagram contains $1$-handles--see \cite{Gompf}). \smallskip

The diagram we draw should specify the Legendrian attaching knots in $S^3$ along with the pairs of $3$-balls indicating the attachments of the $1$-handles. By removing a point away from these attachments, we reduce the picture in $S^3$ to a picture in $\R^3$. After a contactomorphism, the contact structure on $\R^3$ is $\ker(dz-ydx)$ in coordinates $(x,y,z)$. The \emph{front projection} is the map $\pi: \R^3\to \R^2$ with $\pi(x,y,z)=(x,z)$. A Legendrian curve in this contact structure is tangent to the contact planes, which happens precisely when the $y$-coordinate is equal to the slope $\frac{dz}{dx}$ of the front projection. Therefore, Legendrian knots can be recovered from their front projections with the requirement that the diagram has no vertical tangencies (instead it will have cusp singularities where the knot is tangent to the fibers of the projection) and the crossings are always resolved so that the over-strand is the strand with the more negative slope (we orient the $y$-axis into the page to maintain the standard orientation convention for $\R^3$ so the over-strand is the strand with a more negative $y$-coordinate). In these front projections, the contact framing $tb$ can be computed combinatorially in terms of the oriented crossings and cusps of the diagram when the diagram is placed in a standard form where the pairs of $3$-balls giving the attaching regions of $1$-handles are related by a reflection across a vertical axis. Namely, $tb$ of a Legendrian knot is the difference of the writhe of the knot and half the number of cusps in the front projection. For a thorough introduction to Legendrian knots, see \cite{ETNYRE2005105}.\smallskip

The set of moves that relate Weinstein handlebody diagrams in Gompf standard form for equivalent Weinstein domains includes Legendrian Reidemeister moves (including how they interact with the $1$-handles) listed in \cite{Gompf}, see Figures \ref{Reid} and \ref{Gompf456}, as well as \textit{handle slides}, and \textit{handle pair cancellations and additions}, shown in \cite{DingGeiges}. Given two $k$-handles, $h_1$ and $h_2$, a \textit{handle slide} of $h_1$ over $h_2$ is given by isotoping the attaching sphere of $h_1$, and pushing it through the belt sphere of $h_2$. We depict a 1-handle slide (along with intermediate Reidemeister and Gompf moves) in Figure \ref{1hslide} and a 2-handle slide in Figure \ref{2hslide}. A 1-handle $h_1$ and a 2-handle $h_2$ can be cancelled, provided that the attaching sphere of $h_2$ intersects the belt sphere of $h_1$ transversely in a single point. We call this a \textit{handle cancellation} and the pair of handles a \textit{cancelling pair}. Likewise a cancelling pair can be added to a Weinstein handlebody diagram, as depicted in Figure \ref{2hcancel}. When multiple 2-handles intersect a single 1-handle, the simplification in Figure \ref{2handles} can be performed to reduce the overall complexity of a Weinstein diagram.\smallskip

\begin{figure}[h!]
	\begin{center}
		\includegraphics[width=8cm]{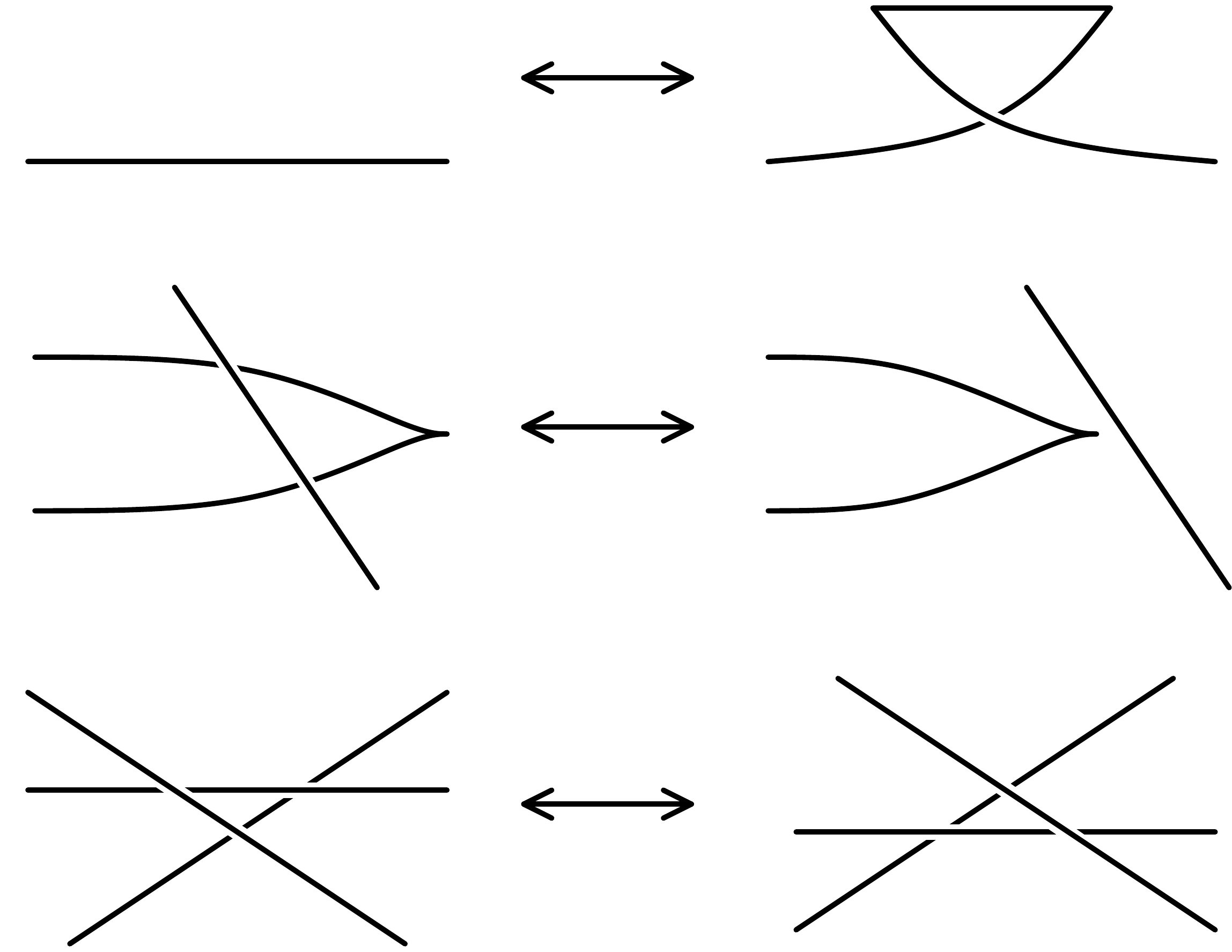}{}
		\caption{The Legendrian Reidemeister moves, up to 180 degree rotation about each axis, where the top, middle, and bottom moves are called Reidemeister I, Reidemeister II, and Reidemeister III, respectively.}
		\label{Reid}
	\end{center}
\end{figure}

\begin{figure}[h!]
	\begin{center}
		\includegraphics[width=12cm]{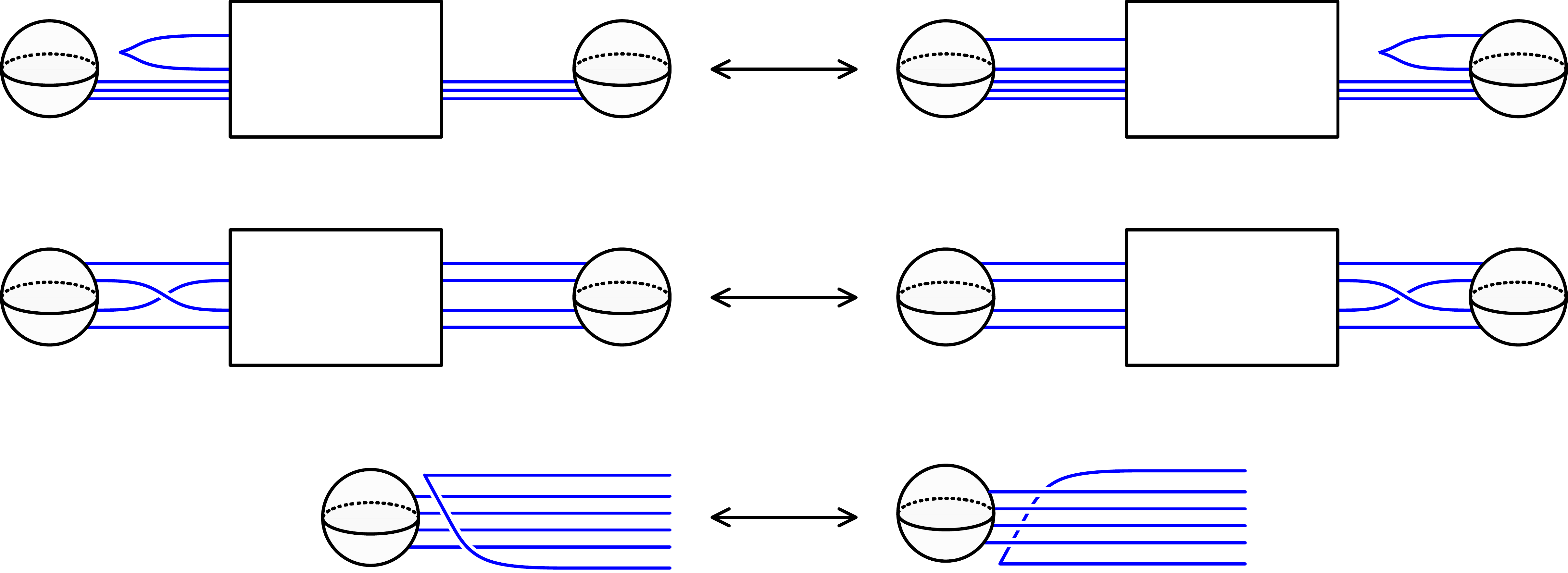}
		\caption{Gompf's three additional isotopic moves, up to 180 degree rotation about each axis. The top, middle, and bottom moves are called Gompf move 4, Gompf move 5, and Gompf move 6, respectively.}
		\label{Gompf456}
	\end{center}
\end{figure}

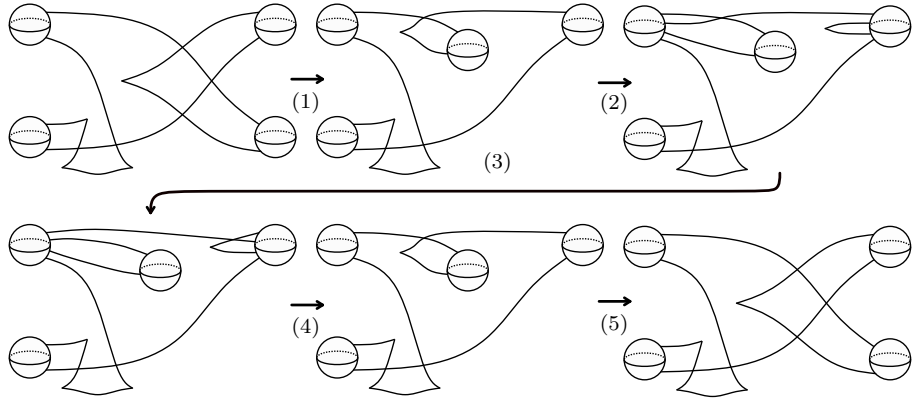
\begin{figure}[h!]
	\begin{center}{}
	\begin{tikzpicture}
			\node[inner sep=0] at (0,0) {\includegraphics[width=12cm]{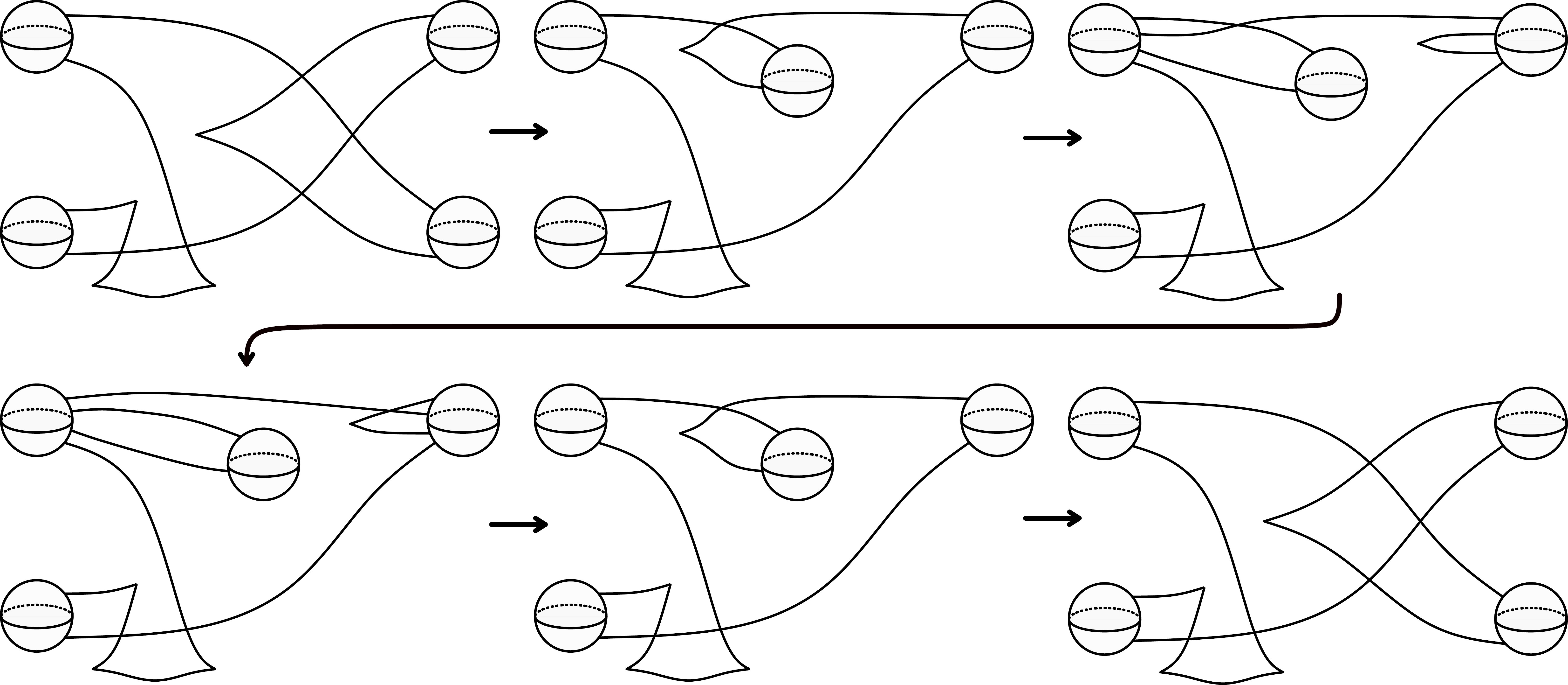}};
			\node at (-2.05,1.3){\footnotesize (1)};
			\node at (2.05,1.28){\footnotesize (2)};
			\node at (0.5,0.5){\footnotesize (3)};
			\node at (-2.05,-1.7){\footnotesize (4)};
			\node at (2.05,-1.65){\footnotesize (5)};
		\end{tikzpicture}
		\caption{An example of a 1-handle slide on $T^*T^2$ consisting of the following sequence of isotopies: (1) Reidemeister 2, (2) Gompf move 4, (3) Gompf move 5, (4) Gompf move 4, (5) Reidemeister 2.}
		\label{1hslide}
	\end{center}
\end{figure}

\begin{figure}[h!]
	\begin{center}
		\includegraphics[width=7cm]{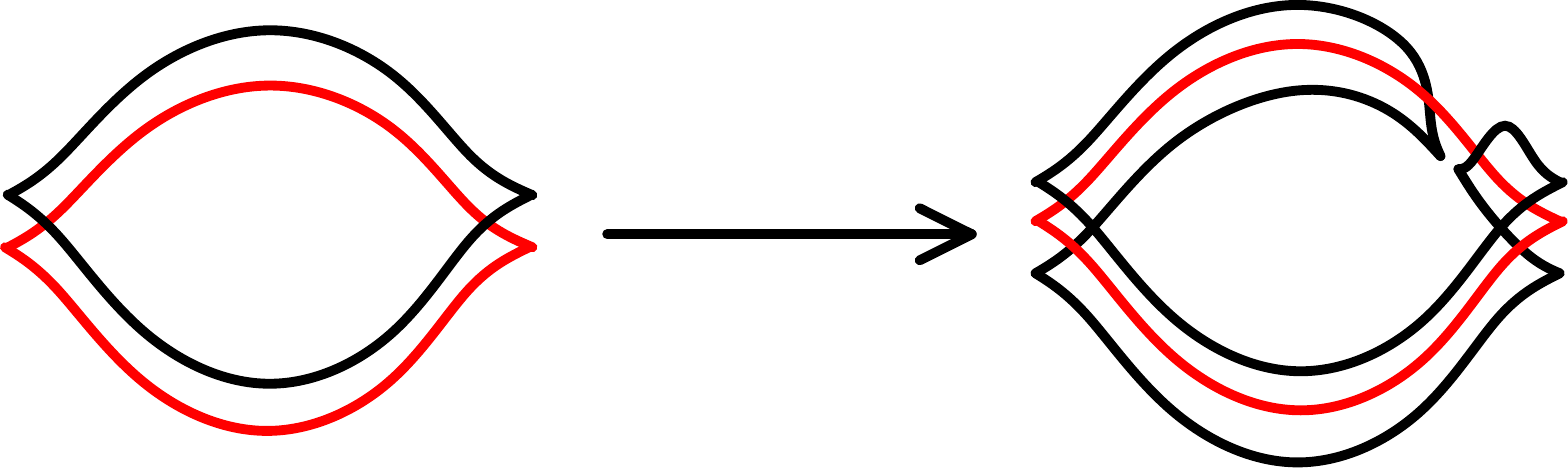}
		\caption{An example of a 2-handle slide of the black unknot over the red unknot.}
		\label{2hslide}
	\end{center}
\end{figure}{}

\begin{figure}[h!]
	\begin{center}
		\includegraphics[width=11cm]{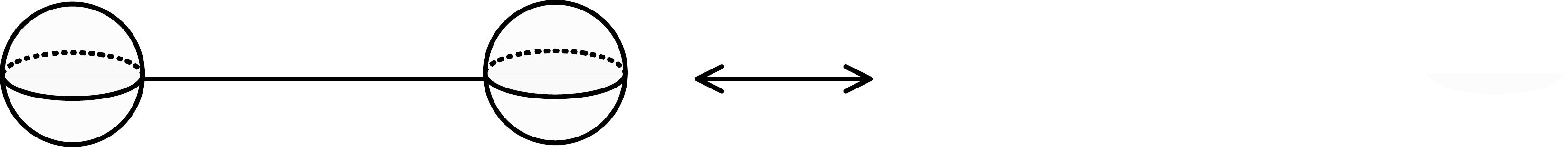}
		\caption{An example of a 1-handle cancelling with a 2-handle.}
		\label{2hcancel}
	\end{center}
\end{figure}{}

\begin{figure}[h!]
	\begin{center}
		\includegraphics[width=8cm]{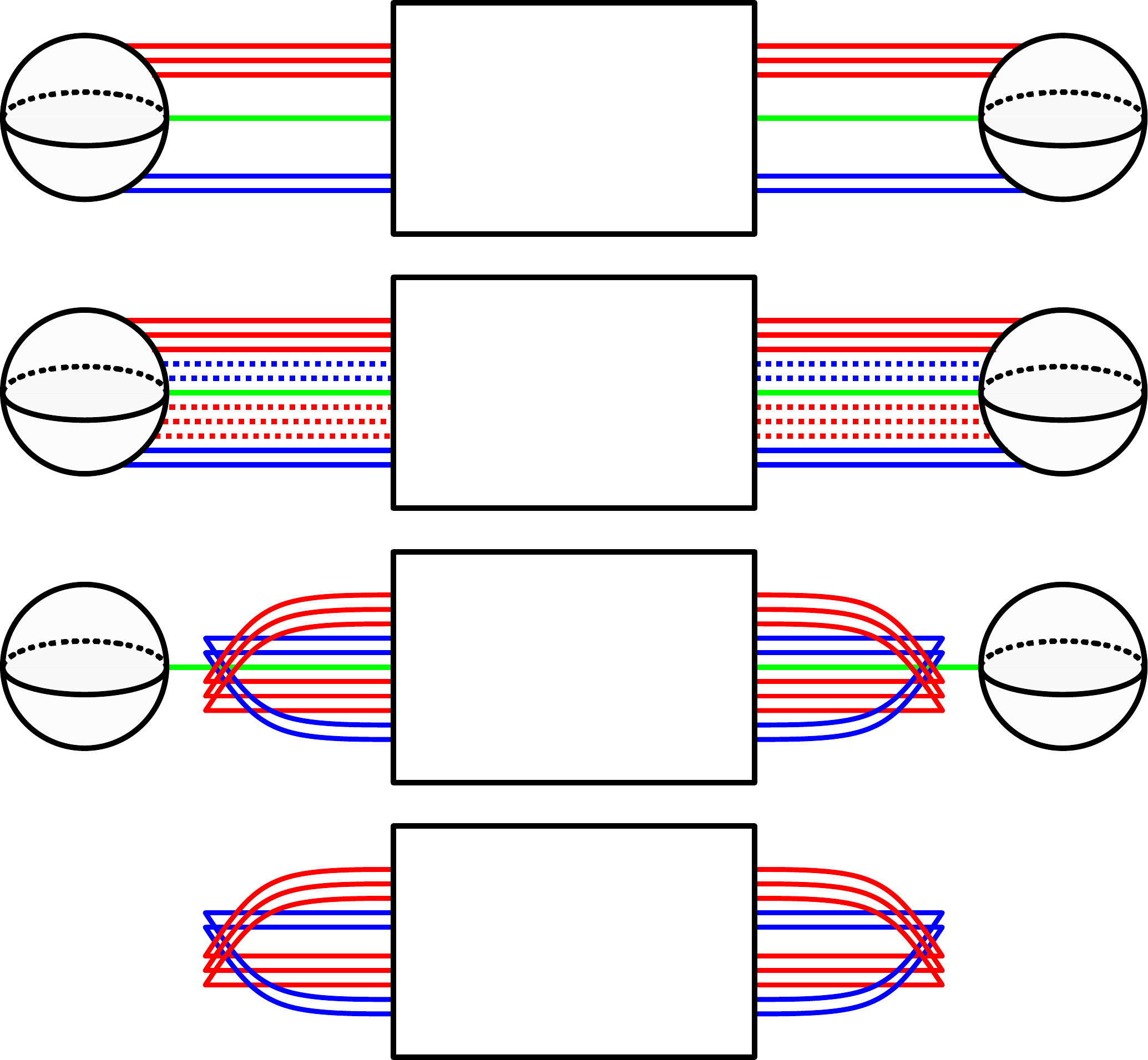}
		\caption{An example of handle slides and cancellations when multiple 2-handles pass through a 1-handle. Red and blue 2-handles are slid over the central green 2-handle. The green 2-handle is then cancelled with the 1-handle.}
		\label{2handles}
	\end{center}
\end{figure}{}

Before approaching our goal of presenting an algorithm to construct Weinstein--Kirby diagrams for complements of smoothed toric divisors, we will start with the unsmoothed case, where the complement has a Liouville completion, $T^*T^2$. The Legendrian handlebody we present, was originally found by Gompf \cite{Gompf}. It follows from that article that this handlebody gives a Stein/Weinstein structure on the smooth manifold $D^*T^2$ (which is the trivial bundle $D^2\times T^2$). More generally, Stein handlebody diagrams are given on the smooth manifolds $D^*\Sigma$ for any surface in \cite{Gompf}. In \cite[Theorem 7.1]{ACGMMSW}, we show that the Weinstein structures induced on these diagrams are Weinstein homotopic to the canonical co-tangent Weinstein structure on $D^*\Sigma$.

For $T^*T^2$ specifically, it is known that there is a unique Weinstein fillable contact structure on the boundary $T^3$ and one can then deduce that the Gompf handlebody agrees with the canonical symplectic structure by Wendl's result that $S^*T^2$ has a unique Stein/Weinstein filling up to deformation \cite{Wendl}. To see that the diagram in Figure \ref{TstarT2} represents $D^*T^2\cong D^2\times T^2$ smoothly, we can start with a handle decomposition for $T^2$ with one 0-handle, two 1-handles and a single 2-handle. Thickening this diagram to a $4$-dimensional handlebody yields a disk bundle over $T^2$ with Euler number $e$, agreeing with the framing coefficient of the $2$-handle attachment. One then needs to put the diagram into Gompf standard form as seen in Figure \ref{stdTstarT2} by sliding the uppermost attaching ball below the attaching ball on the right so that both $1$-handles are related by a reflection across the same vertical axis. Then we must realize the knot as a Legendrian knot by replacing vertical tangencies by cusps and making sure the crossings always have the over-strand corresponding to the more negative slope. The most obvious way to do this yields a Legendrian knot whose Thurston-Bennequin framing is $0$ (see the diagram on the right of Figure \ref{stdTstarT2}), so this would correspond to a $D^2$ bundle over $T^2$ with Euler number $-1$. By wrapping one strand around the lower left attaching ball as in Figure \ref{TstarT2}, we obtain a smoothly isotopic picture where the new Legendrian has $tb=1$, so the Euler number is $1-1=0$ as needed for $D^*T^2$. Since this is one Weinstein filling of $S^*T^2$, and we know that such fillings are unique up to deformation, it must agree with the canonical co-tangent symplectic structure on $D^*T^2$.

\begin{figure}[h!]
    \centering
    \subfigure{\includegraphics[width=0.25\textwidth]{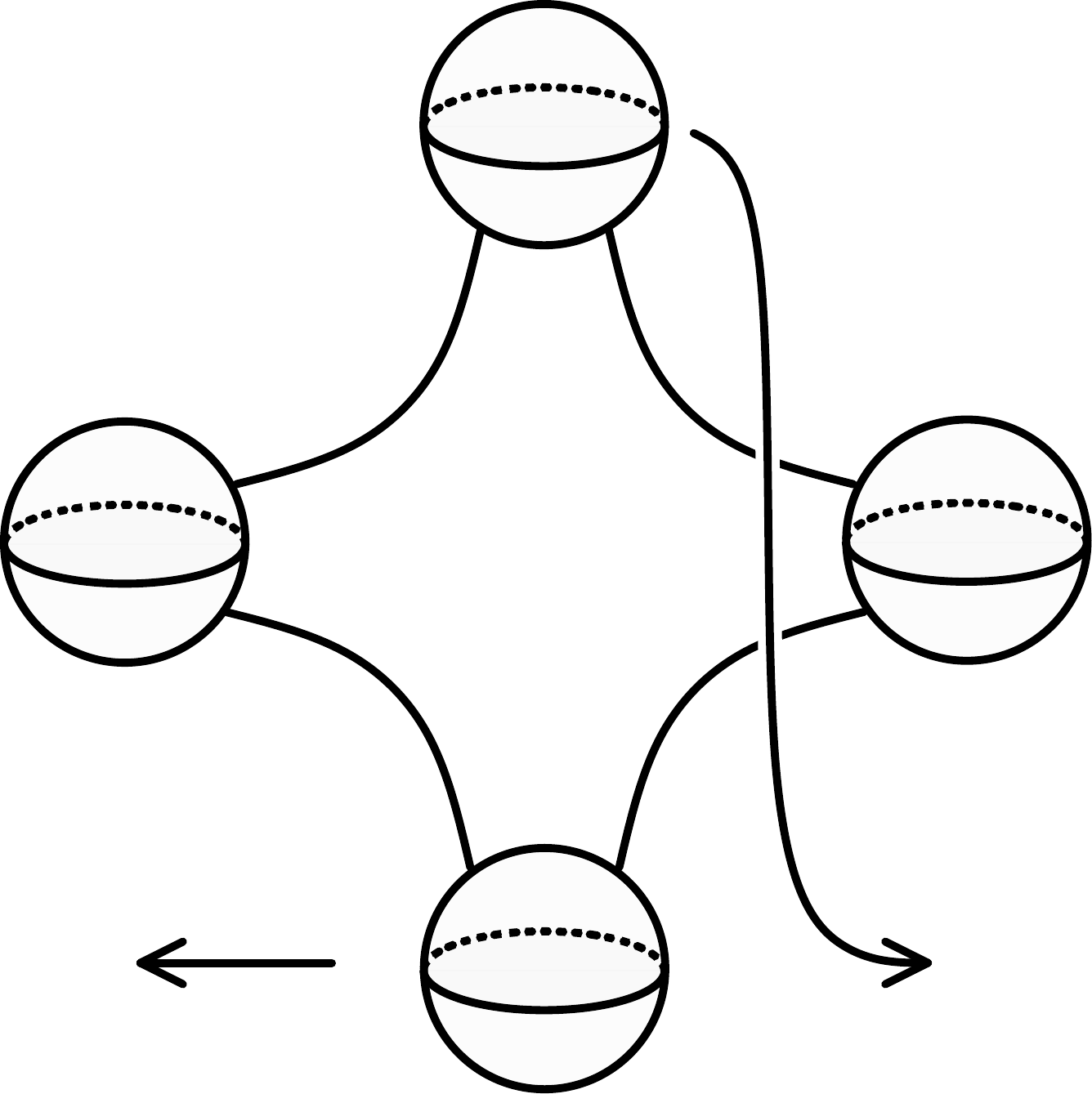}} \hspace{2.5cm}
    \subfigure{\includegraphics[width=0.30\textwidth]{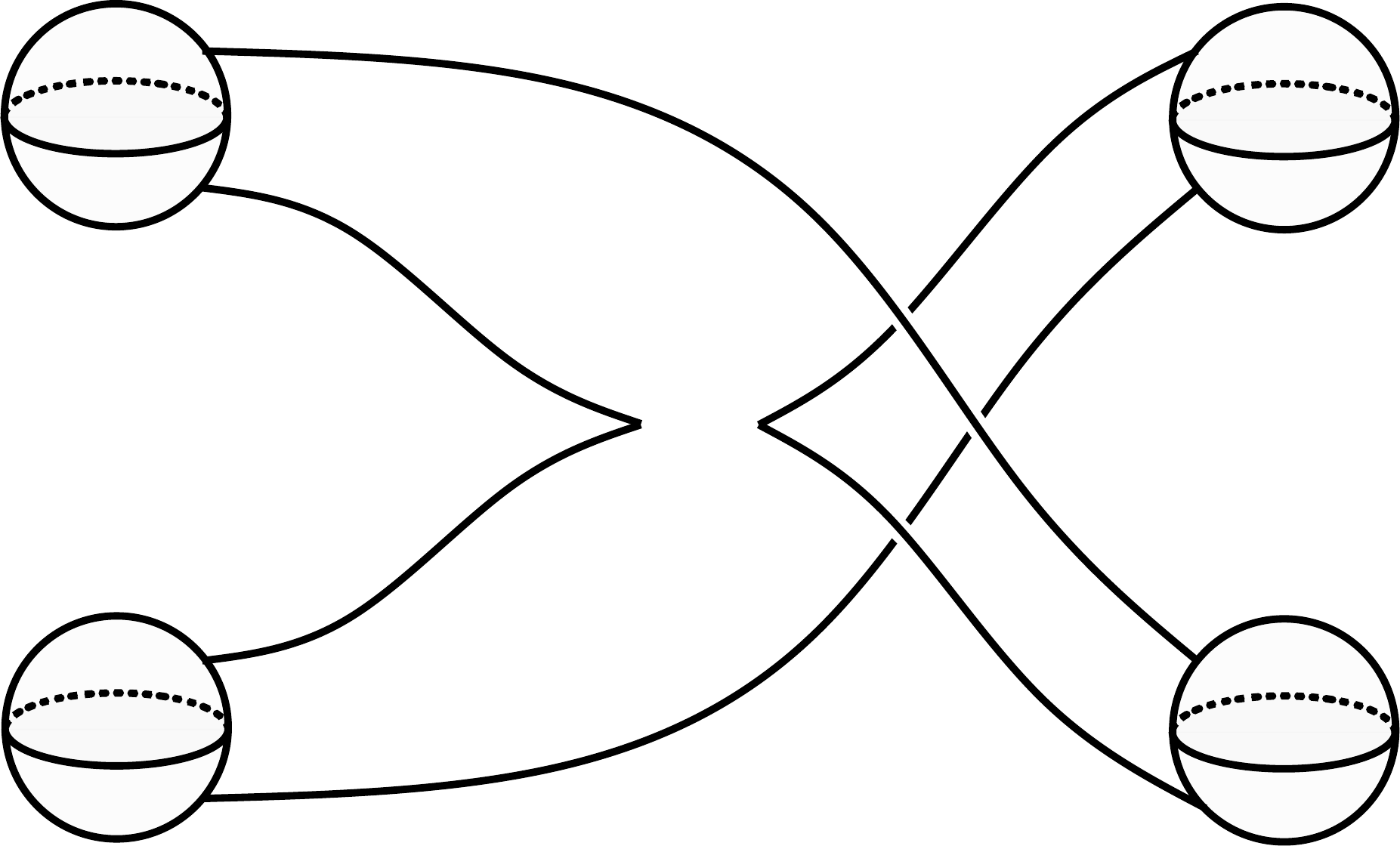}}
        
    \caption{Diagram depicting how to move the usual picture of $T^*T^2$ into standard form after inserting the necessary Legendrian data, i.e. replacing vertical tangencies by cusps.}
    \label{stdTstarT2}
\end{figure}

\begin{figure}[h!]
	\begin{center}
		\includegraphics[width=6cm]{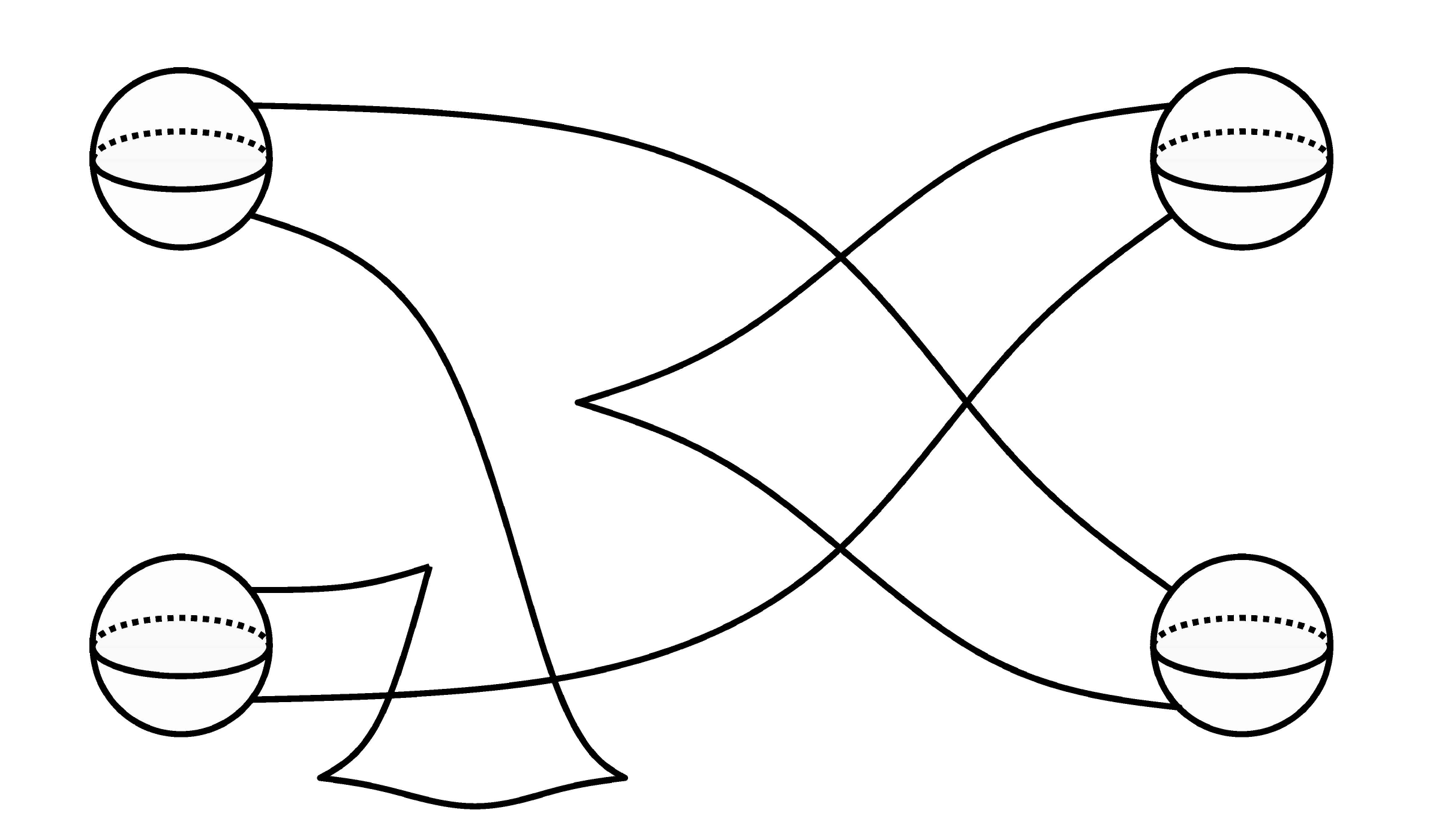}
		\caption{Stein structure on $D^2$ bundle over $T^2$ with framing coefficient $e(T^*T^2) \leq 0$.}
		\label{TstarT2}
	\end{center}
\end{figure}


\section{The local model for our handle attachment} \label{s:handle}

In this section we study the local model for the smoothing of a normal crossing singularity of a symplectic 
divisor in dimension~$4$ and describe the local handle attachment information (core and co-core) for the $2$-handle attachment needed to describe the complement of the smoothing of a node. Here we give the intuitive picture behind the following theorem.

\begin{theorem}[Theorem 4.1~\cite{ACGMMSW}]\label{thm:toricWein}
	Let $(M,\omega)$ be a toric $4$-manifold corresponding to Delzant polytope $\Delta$ which is $\{V_1,\ldots, V_n \}$-centered. Let $D$ denote the divisor obtained by smoothing the toric divisor at the nodes $V_1,\ldots, V_n$. Then there exist arbitrarily small neighborhoods $N$ of $D$ such that $M\setminus N$ admits the structure of a Weinstein domain.
	
	Furthermore, $M\setminus N$ is Weinstein homotopic to the Weinstein domain obtained by attaching Weinstein $2$-handles to the unit disk cotangent bundle of the torus $D^*T^2$, along the Legendrian co-normal lifts of co-oriented curves of slope $s(V_1),\ldots, s(V_n)$. Here $s(V_i)$ is equal to the difference of the inward normal vectors of the edges adjacent to $V_i$ in $\Delta$.
\end{theorem}

A local model in a $4$-dimensional manifold~$M$ for the normal crossing intersection of two 
symplectic divisors can be given by a Darboux chart~$(\C^2, \omega_{\std})$ at the 
intersection point, where the two divisors are mapped to the two axes 
$\Sigma_1 = \left \{{(z_1,z_2) \in \C^2 \, | \, z_2 = 0} \right\}$ and 
$\Sigma_2 = \left \{{(z_1,z_2) \in \C^2 \, | \, z_1 = 0} \right\}$. 
Smoothing this normal crossing means that locally one substitutes the union of these divisors 
by the smooth surface 
$$\Sigma = \left \{{(z_1,z_2) \in \C^2  \, | \,  z_1 \cdot z_2 = \varepsilon^2} \right\}$$ 
for some $\varepsilon > 0$. Topologically, the complement of the smoothed surface differs from the complement of the normal crossing divisors by one 2-handle attachment for each smoothed normal crossing intersection. Under the centered hypothesis (Definition~\ref{def:centered}), Theorem~\ref{thm:toricWein} says that when considering the symplectic and Liouville structures, the difference between the complements also corresponds to a collection of \emph{Weinstein} 2-handle attachments.

In order to encode this handle attachment in the Weinstein Kirby diagram, we need to 
identify the corresponding Legendrian attaching sphere. We find natural Lagrangian representatives of the co-core and core of the handle in the complement of the smoothed divisor. From there, the attaching sphere is the boundary of the core.

A co-core of a Weinstein $2$-handle is characterized (uniquely up to Lagrangian isotopy \cite{EliPol}) as a Lagrangian disk contained in the smooth $2$-handle with unknotted Legendrian boundary lying in the boundary of the handle (not the attaching region). In our model, the boundary of the handle is in the boundary of the small neighborhood of the smoothed hypersurface $\Sigma$. 
One can consider, for $0 < \varepsilon' < \varepsilon$, the disk $D_1$ defined as the image of the map
$$
\begin{array}{ccccc}
\phi & : & [0,1] \times [0,2\pi] & \rightarrow & \C^2\\
&   & (r,\theta)                   & \mapsto    & \left({ \begin{array}{c} 
	r \varepsilon' e^{i \theta}\\ 
	r \varepsilon' e^{-i \theta}
	\end{array}}\right)
\end{array}
$$
The size of $\varepsilon'$ is determined by the size of the neighborhood of the smoothed divisor which is deleted. As $\varepsilon'\to \varepsilon$, the open disk of radius $\varepsilon$ provides the analog of the co-core in the non-compact complement of the divisor itself.

This disk is Lagrangian as for $r \neq 0$, the derivative of $\phi$ 
$$d_{(r,\theta)} \phi = \left({ 
	\begin{array}{cc}
	\varepsilon' e^{i \theta} & i r \varepsilon' e^{i \theta}\\
	\varepsilon' e^{-i \theta} & -i r \varepsilon' e^{-i \theta}
	\end{array}
}\right)
$$
is an isomorphism so that the tangent space to $D_1$  at the point $\phi(r,\theta)$ is spanned 
by the vectors
$$u_1 = \left({ \begin{array}{r} 
	\varepsilon' e^{i \theta}\\ 
	\varepsilon' e^{-i \theta}
	\end{array}}\right)                       
\textrm{ and }
u_2 = \left({ \begin{array}{r} 
	i r \varepsilon' e^{i \theta}\\ 
	-i r \varepsilon' e^{-i \theta}
	\end{array}}\right)$$
and one can check that 
$$\omega_{\std}(u_1, u_2) = \Im m ( \langle u_1 ,  u_2 \rangle )  = 0$$
where $\langle \cdot ,  \cdot \rangle$ is the standard Hermitian product in $\C^2$ 
(equivalently one can check that $\phi^*(\frac{i}{2}(dz_1\wedge d\bar{z}_1+dz_2\wedge d\bar{z}_2))=0$). The point for $r=0$ corresponds to the origin of $\C^2$. At this point, the two following curves $c_1$ 
and $c_2$ in $D_1$ parametrized for $t \in (-1,1)$ by:
$$c_1(t) = \left({ \begin{array}{c} 
	t \varepsilon'\\ 
	t \varepsilon'
	\end{array}}\right)
\textrm{ and }      
c_2(t) = \left({ \begin{array}{c} 
	i t \varepsilon'\\ 
	-i t \varepsilon'
	\end{array}}\right)$$
give the two independent vectors in the tangent space at the origin of the disk~$D_1$:
$$u_1 = \left({ \begin{array}{c} 
	\varepsilon' \\ 
	\varepsilon' 
	\end{array}}\right)                       
\textrm{ and }
u_2 = \left({ \begin{array}{c} 
	i  \varepsilon' \\ 
	-i  \varepsilon'
	\end{array}}\right)$$
One can note that we have again $\omega_{\std}(u_1, u_2)  = 0$.

The boundary of this disk is the image by $\phi$ of $\{1\} \times [0,2\pi],$ that is, the circle 
$$B = \left\{ { \left. {\left({ \begin{array}{c} 
			\varepsilon' e^{i \theta}\\ 
			\varepsilon' e^{-i \theta}
			\end{array} }\right) } \right| \theta \in [0, 2\pi] }\right\}.$$
			
One can choose $\varepsilon'<\varepsilon$ such that the circle $B$ lies on the boundary of the small neighborhood of the smoothed $\Sigma$ (since $z_1 z_2 = \varepsilon' e^{i \theta} \varepsilon' e^{-i \theta} = \varepsilon'^2$ goes to $\varepsilon^2$ if $\varepsilon'$ approaches $\varepsilon$). It limits to the origin when $\varepsilon'$ goes to $0$, so that $B$ is the belt sphere of the handle and $D_1$ is indeed its co-core.\smallskip

The core is characterized (uniquely up to Lagrangian isotopy) as a Lagrangian disk with unknotted boundary in the smooth $2$-handle which intersects the co-core transversally at one point and which avoids the boundary of the Weinstein manifold, so in our model, it should avoid the smoothed~$\Sigma$.

Let $D_2$ be the disk defined as the image of the map
$$
\begin{array}{ccccc}
\psi & : & [0,1] \times [0,2\pi] & \rightarrow & \C^2\\
&   & (r,\theta)                   & \mapsto    & \left({ \begin{array}{c} 
	r \varepsilon e^{i \theta}\\ 
	r \varepsilon e^{-i (\theta+\pi)}
	\end{array}}\right)
=  \left({ \begin{array}{c} 
	r \varepsilon e^{i \theta}\\ 
	-r \varepsilon e^{-i \theta}
	\end{array}}\right)
\end{array}
$$

This disk is also Lagrangian as, 
similarly as before, 
the tangent space to $D_2$  at the point $\psi(r,\theta)$ for $r \neq 0$ is spanned 
by the vectors
$$u_3 = \left({ \begin{array}{r} 
	\varepsilon e^{i \theta}\\ 
	-\varepsilon e^{-i \theta}
	\end{array}}\right)                       
\textrm{ and }
u_4 = \left({ \begin{array}{r} 
	i r \varepsilon e^{i \theta}\\ 
	i r \varepsilon e^{-i \theta}
	\end{array}}\right)$$
and one can check that $$\omega_{\std}(u_3, u_4) = \Im m ( \langle u_3 ,  u_4 \rangle )  = 0.$$

Similarly, at the origin, the two following curves $c_3$ and $c_4$ in $D_2$ parametrized for $t \in (-1,1)$ by
$$c_3(t) = \left({ \begin{array}{c} 
	t \varepsilon\\ 
	-t \varepsilon
	\end{array}}\right)
\textrm{ and }      
c_4(t) = \left({ \begin{array}{c} 
	i t \varepsilon\\ 
	i t \varepsilon
	\end{array}}\right)$$
give the two independent vectors in the tangent space at the origin of the disk~$D_2$:
$$u_3 = \left({ \begin{array}{c} 
	\varepsilon \\ 
	-\varepsilon 
	\end{array}}\right)                       
\textrm{ and }
u_4 = \left({ \begin{array}{c} 
	i  \varepsilon \\ 
	i  \varepsilon
	\end{array}}\right)$$
Note again that $\omega_{\std}(u_3, u_4)  = 0$. This disk does not intersect the smoothed  $\Sigma$ as
$$z_1 z_2 = r \varepsilon e^{i \theta} (- r \varepsilon e^{-i \theta}) = - r^2 \varepsilon^2 \neq \varepsilon^2.$$ Moreover, $D_1$ and $D_2$ intersect at the origin and this intersection is transverse as one can check that the family $(u_1, u_2, u_3, u_4)$ spans $\C^2$ as a real vector space. This shows that $D_2$ is the core of the handle and the attaching sphere is  the image by $\psi$ of $\{1\} \times [0,2\pi]$, that is, $$A = \left\{ { \left. {\left({ \begin{array}{c} 
			\varepsilon e^{i \theta}\\ 
			\varepsilon e^{-i (\theta+\pi)}
			\end{array} }\right) } \right| \theta \in [0, 2\pi] }\right\}.$$

Identifying these transverse Lagrangian disks with the core and co-core in Weinstein's model for a $4$-dimensional $2$-handle, allows us to place a Weinstein structure (a Liouville vector field and corresponding Morse function) on this new piece which lies in the complement of the regular neighborhood of the smoothing. If the polytope is centered with respect to all the nodes we are smoothing, the Weinstein structures on these pieces glue together consistently with a Weinstein structure on the complement of the unsmoothed toric divisor (a Weinstein domain which completes to $T^*T^2$). It is shown in~\cite[Section 4.5]{ACGMMSW} how these fit together to give an explicit global Weinstein structure.			
			
In the toric description of the toric manifolds and divisors we consider, the Hamiltonian torus action in the local Darboux model corresponds to the torus action on $\C^2$ given in coordinates: $(e^{i\theta_1}, e^{i\theta_2}) \ast (z_1,z_2)=(e^{i\theta_1}z_1,e^{i\theta_2}z_2)$. In particular, through the symplectomorphism between the complement of the normal crossing divisor we consider and~$T^*T^2$, the orbit of a point corresponds to the torus 
$T^2$ and the quotient space under the Hamiltonian action corresponds to a cotangent 
fiber. The attaching sphere in the model corresponds in this symplectic identification 
to a lift of a circle of slope~$(1, -1)$ in the base~$T^2$ to the 
cotangent bundle~$T^*T^2$ (lift corresponding to the 
point~$(\varepsilon^2, \varepsilon^2)$ in the quotient space $(\R_{>0})^2$). 
The standard model neighborhood corresponds to the standard cone $\R_{\geq 0}\times \R_{\geq 0}$, and the corresponding attaching slope is $(1,-1)=(1,0)-(0,1)$, the difference of the inward normal vectors.
In general, relating any node to the standard model via an $SL(2,\Z)$ transformation, the attaching sphere for the $2$-handle corresponding to the smoothing of a normal crossing singularity at a chosen vertex $V$ will be the co-normal lift of a circle of slope $s(V)$, where $s(V)$ is the difference of the inward normal vectors of the edges adjacent to $V$, 
(see~\cite[Lemma 4.2]{ACGMMSW}).


\section{The algorithm through an example}

Here we will show how to obtain a Weinstein handle diagram for the complement of a toric divisor with exactly one node smoothed. Note that for any Delzant polytope, and any single vertex $V$, the polytope is vacuously $\{V\}$-centered. To fix a specific example, we could start with the toric $4$-manifold $\cptwo$ whose toric divisor is a collection of three $\cpone$'s intersecting at three points, and smooth this divisor at one of the intersection points. More explicitly in homogeneous coordinates $[z_0:z_1:z_2]$ on $\cptwo$, the toric divisor is given by the union of the three lines $L_0=\{z_0=0\}$, $L_1=\{z_1=0\}$ and $L_2=\{z_2=0\}$. Let us smooth the intersection of $L_1$ with $L_2$ at $[1:0:0]$. In the affine coordinate chart where $z_0=1$, with coordinates $(z_1,z_2)$, this aligns exactly with our local model in Section \ref{s:handle}. After smoothing, the lines $L_1$ and $L_2$ are joined to form a conic $Q$, which intersects the remaining line $L_0$ at two points. By Theorem~\ref{thm:toricWein}, the complement of this smoothed divisor is obtained by attaching a single $2$-handle to $D^*T^2$ along the Legendrian lift of a curve in $T^2$ of slope $(1,-1)$.

	\begin{figure}[!ht]
		\centering
		\includegraphics[scale=.3]{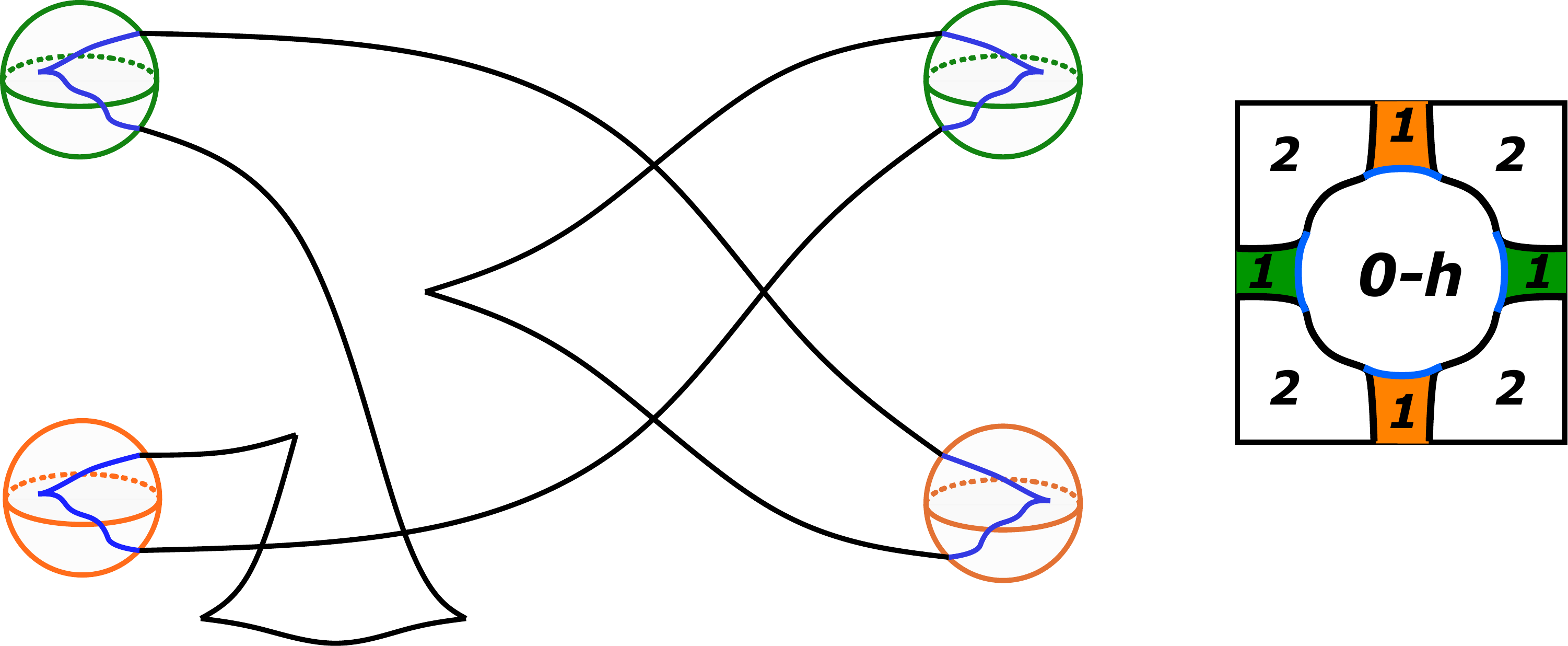}
		\caption{Left: The Legendrian unknot $\mathcal{K}$ in the boundary of the $4$-dimensional $0$-handle of the Gompf diagram, indicating the intersection of this boundary with the Lagrangian torus ($0$-section of $D^*T^2$). The black portions coincide with segments of the attaching circle of the $2$-handle, and the blue portions give the attaching arcs of the $2$-dimensional $1$-handles of the torus. Right: The corresponding decomposition of the Lagrangian torus.}
		\label{fig:torusdecomp}
	\end{figure}
	
	In order to translate Legendrian attaching circles in $S^*T^2$ described as the co-normal lift of a curve in $T^2$ into Legendrian curves drawn in the Gompf diagram (Figure \ref{TstarT2}), we need to understand how these two pictures get identified. As mentioned in Section \ref{s:kirby}, the Gompf handle diagram is obtained by starting with a smooth handle decomposition of $T^2$ with a single $0$-handle, two $1$-handles, and one $2$-handle. This diagram is thickened by two dimensions to obtain $T^2\times D^2$, and then the attaching curve of the $2$-handle is isotoped around until it agrees with a Legendrian front diagram with induced framing $tb-1 = 0$. On the other hand, the co-normal lift construction is more compatible with the canonical (Morse-Bott) Weinstein structure on $D^*T^2$ which has critical locus along the $0$-section. In \cite[Theorem 7.1]{ACGMMSW}, we prove that these two structures are Weinstein homotopic and identify the image of the Lagrangian torus giving the zero-section of $D^*T^2$ in the Gompf diagram. The handle decomposition on the $4$-manifold induces the corresponding handle decomposition on the Lagrangian torus by intersection. In particular, we see a Legendrian (un)knot $\mathcal{K}$ in the boundary of the $4$-dimensional $0$-handle, which partially coincides with the attaching sphere of the $2$-handle, and partially corresponds with attaching arcs for the $1$-handles of the Legendrian torus. See Figure \ref{fig:torusdecomp}.

\begin{figure}[!ht]
	\begin{center}
		\includegraphics[width=8cm]{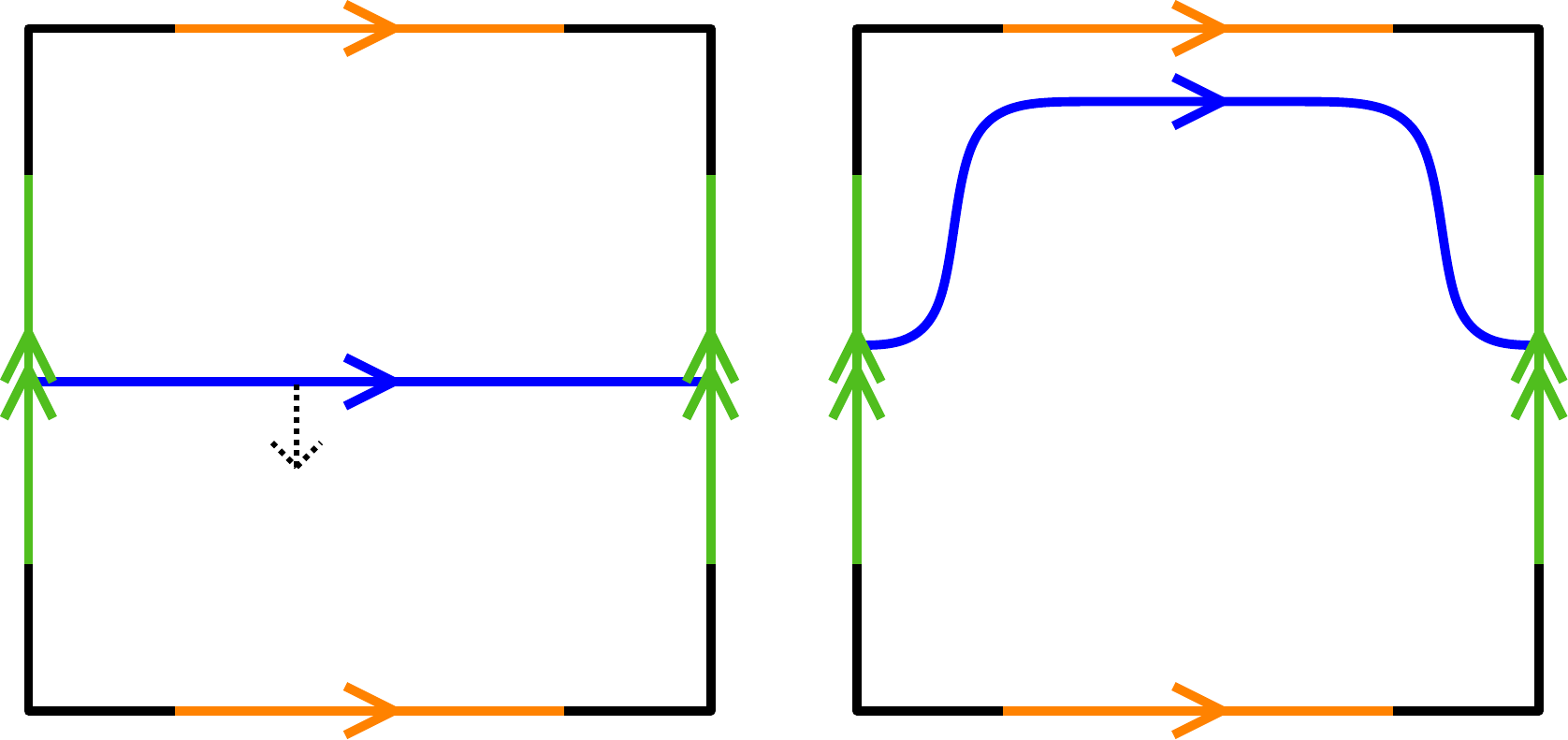}
		\caption{The circle with $(1,0)$-slope on $T^2$ with the dotted arrow indicating the co-orientation of the curve, and an isotoped version which runs parallel to the boundary of the $0$-handle of $T^2$ except where it passes through a $1$-handle.}
		\label{onenodetorus}
	\end{center}
\end{figure}

\begin{figure}[!ht]
	\begin{center}
		\includegraphics[width=10 cm]{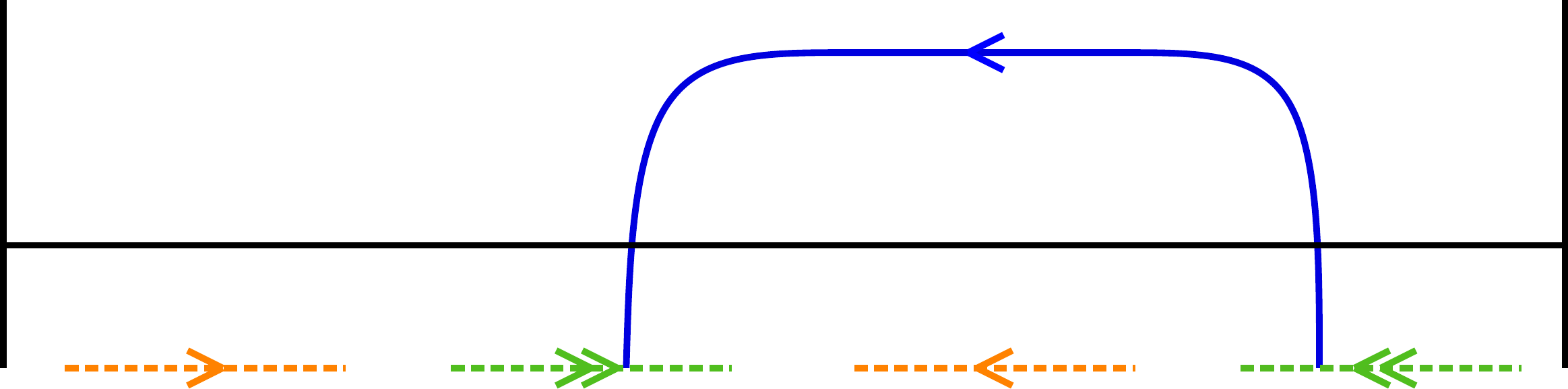}
		\caption{The curve in $J^1(S^1)$ which is identified with the curve $(1,0)$ on $T^2$ as in Figure \ref{onenodetorus}.}
		\label{onenodejet}
	\end{center}
\end{figure}

\begin{figure}[!ht] 
    \centering
    \subfigure[The neighborhood $U$ of the co-normal lift of a curve $\gamma \subset T^2$]{\includegraphics[width=0.17\textwidth]{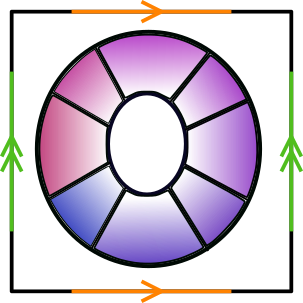}}  \hspace{2cm}
    \subfigure[Identification of $U$ with neighborhood of the zero section of $J^1(S^1)$.]{\includegraphics[width=0.32\textwidth]{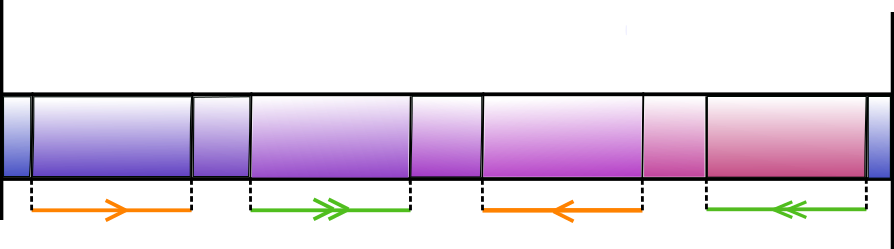}}  
\vspace{1cm}
    \subfigure[The front projection of $U$ in the Gompf diagram for $D^*T^2$]{\includegraphics[width=0.43\textwidth]{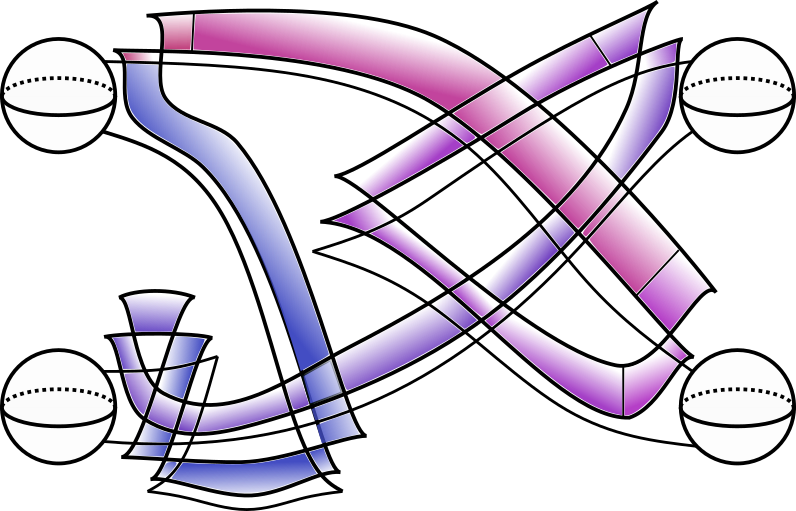}}
\caption{Transferring curves from $T^2$ to $J^1(S^1)$ and then to the Gompf diagram.}
\label{fig:gradient} 
\end{figure}

	Now consider the Legendrian co-normal lift to $S^*T^2$ of the circle in $T^2$ which is the boundary of the $2$-dimensional $0$-handle, with the inward co-orientation. This is a Legendrian push-off of $\mathcal{K}$ in $\partial B^4 = \partial D^*D^2$, because a small positive Reeb flow applied to $\mathcal{K}=\partial D^2$ yields the co-normal lift of a concentric circle close to $\partial D^2$. We will perform an isotopy to the curves in our torus corresponding to attaching spheres of the additional $2$-handles so that these curves agree with parallel copies of such circles except where they enter the $1$-handles. Circles which are further inward will be pushed off more in the \emph{positive} Reeb direction. Note that every Legendrian circle has a standard neighborhood which is contactomorphic to a neighborhood of the zero section in $J^1(S^1)$ with the contact form $dz-ydx$ where $x$ is the coordinate on $S^1$. We will translate the diagram on the torus to a front projection diagram of $J^1(S^1)$, where we think of the $S^1$ as the Legendrian boundary of the $0$-handle of the $T^2$ $0$-section. Since the Reeb direction is the positive $\partial_z$ direction in $J^1(S^1)$, circles which are further inwards in the torus (pushed further by the Reeb flow) will correspond to curves which are pushed upwards more in the $J^1(S^1)$ diagram. See an example of this procedure in Figures \ref{onenodetorus} and \ref{onenodejet}. Finally, once we have our diagram in the 1-jet space, we can satellite the diagram onto the image of $S^1$ in the Gompf diagram. The images of the co-normal lifts of curves contained in a neighborhood $U$ of the curve $\gamma$ that coincides with the attaching circle of the $2$-handle and the attaching arcs of the $2$ dimensional $1$ handles of the torus are illustrated in Figure \ref{fig:gradient}. As indicated by the shading gradient, curves which lie further inward towards the center of the square correspond to curves at greater $z$-height values in the jet-space. In turn, curves at a higher $z$-height in the jet space will correspond to higher Reeb pushoffs of the complicated looking Legendrian unknot of Figure~\ref{fig:gradient}(c).

	Initially, let us apply this procedure in the simple example where we are attaching a $2$-handle along the co-normal lift of a circle in the torus with slope $(1,0)$. The cotangent projection of this model is presented on the left in Figure \ref{onenodetorus}. Push the curve $(1,0)$ to the upper side of the square, so it lies close to the boundary, and then cut the rectangle at the bottom left vertex to map to $J^1(S^1)$, obtaining Figure \ref{onenodejet}. Satelliting this onto the Legendrian unknot in Figure \ref{fig:torusdecomp}, we obtain Figure \ref{onenode}.\smallskip
	
\begin{figure}[!ht]
	\begin{center}
		\includegraphics[width=6cm]{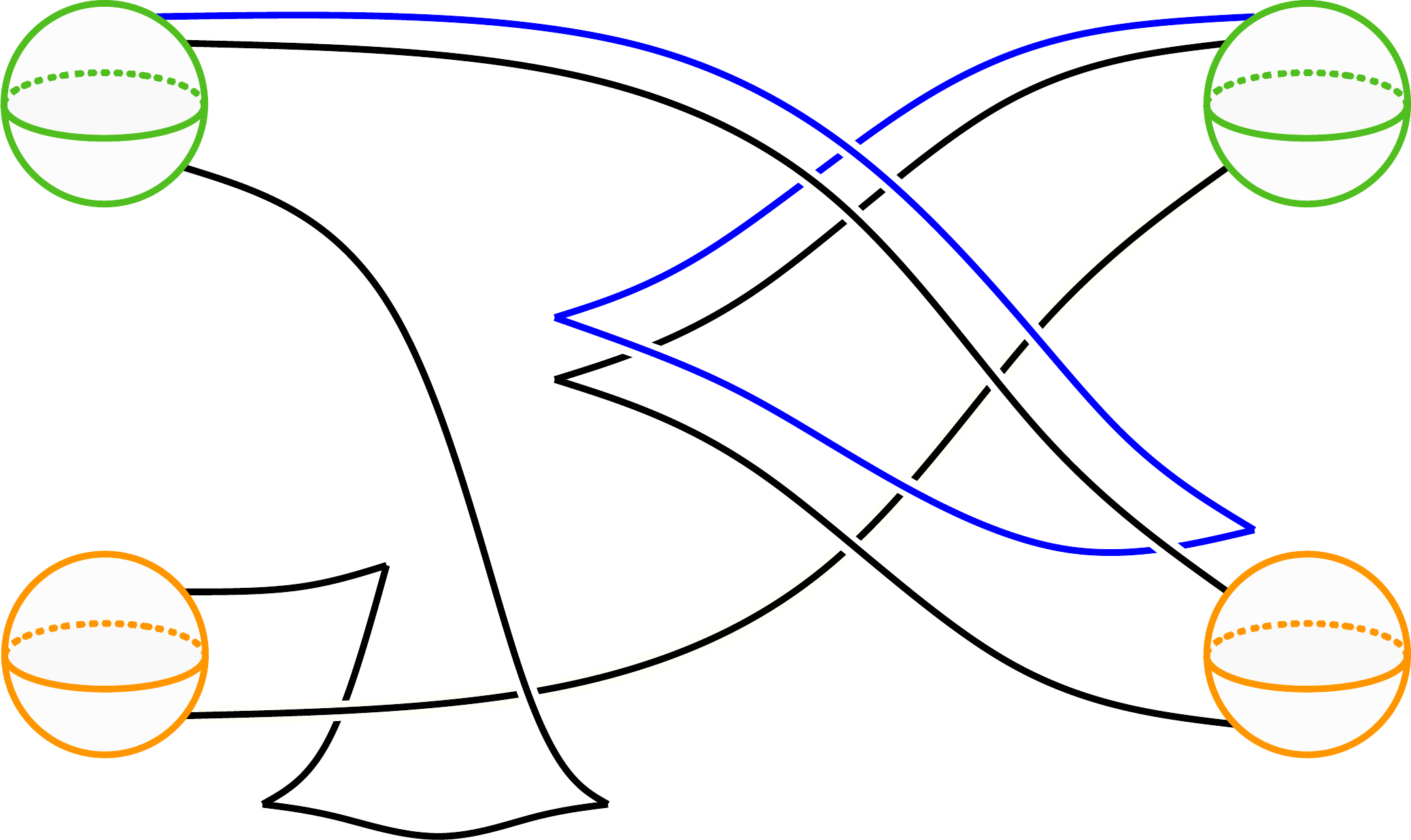}
		\caption{The Legendrian handle diagram of the complement of the toric divisor smoothed in one node. That is, $T^*T^2 \cup \Lambda_{(1,0)}$.}
		\label{onenode}
	\end{center}
\end{figure}

	\begin{figure}[!ht]
		\begin{center}
			\includegraphics[width=8cm]{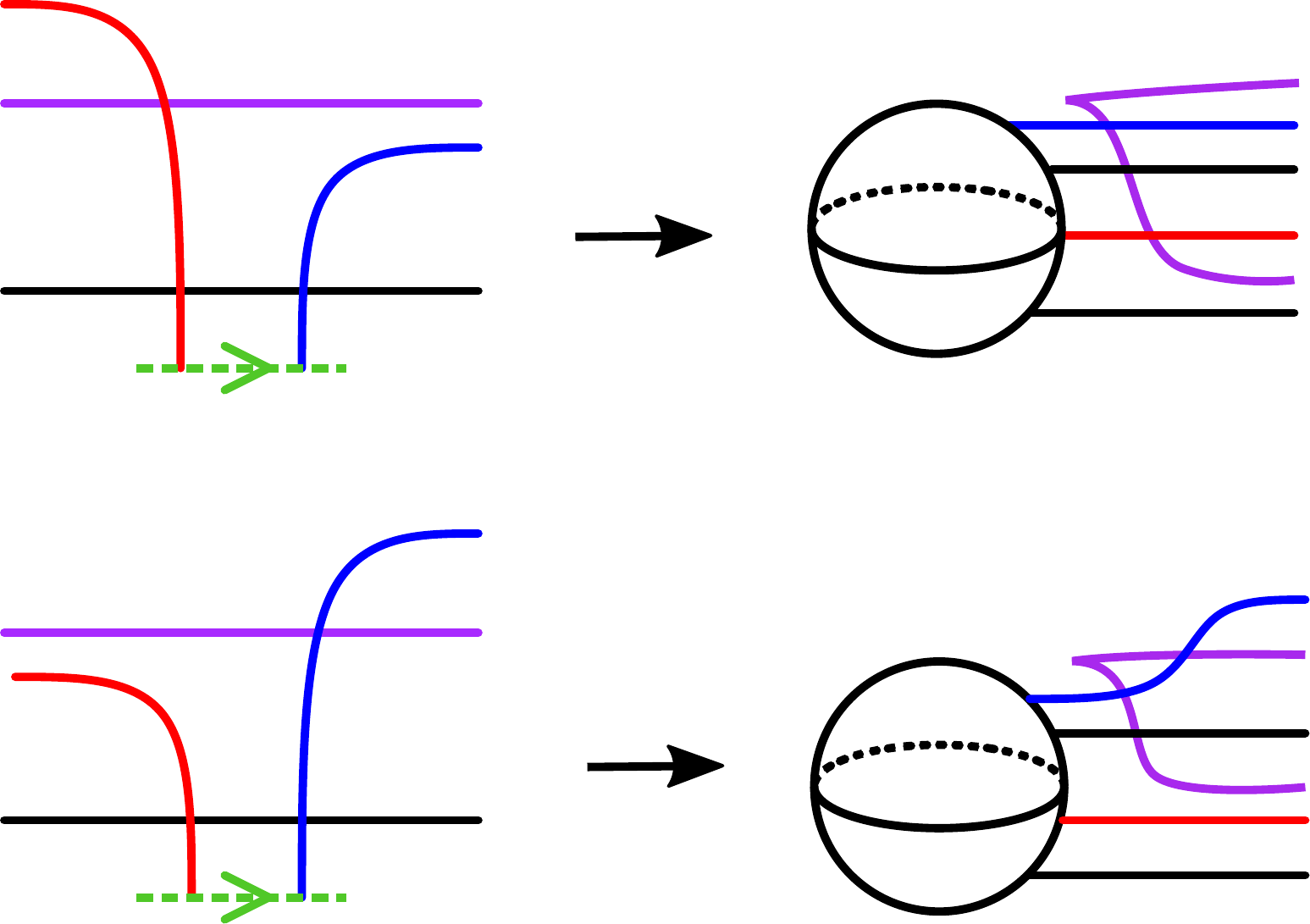}
			\caption{Mapping the red, blue and purple curves from $J^1(S^1)$ to $T^*T^2$. The relative heights of these curves with respect to the black curve are preserved. The purple curve follows a Reeb push off of the cusping curve in the 1-handle's attaching region (see Figure \ref{fig:torusdecomp}) while the red and blue curves pass through the 1-handle.}
			\label{curvecorner}
		\end{center}
	\end{figure}

	In general, when satelliting, we need to be somewhat careful with the behavior of the curves near the $1$-handle. If the curves pass above the attaching region of a $1$-handle without entering the $1$-handle, they will follow an upward Reeb push-off of the cusps that appear inside the $1$-handle attaching balls in Figure \ref{fig:torusdecomp}. Note that we will typically push these cusps out of the attaching regions of the $1$-handles by a Legendrian isotopy. If the curves pass through the $1$-handle in the torus, they will pass through the corresponding $1$-handle in the $4$-dimensional handlebody. See Figure \ref{curvecorner} for the conventions in a more complicated example.\smallskip

\begin{figure}[!ht]
	\begin{center}
		\includegraphics[width=14cm]{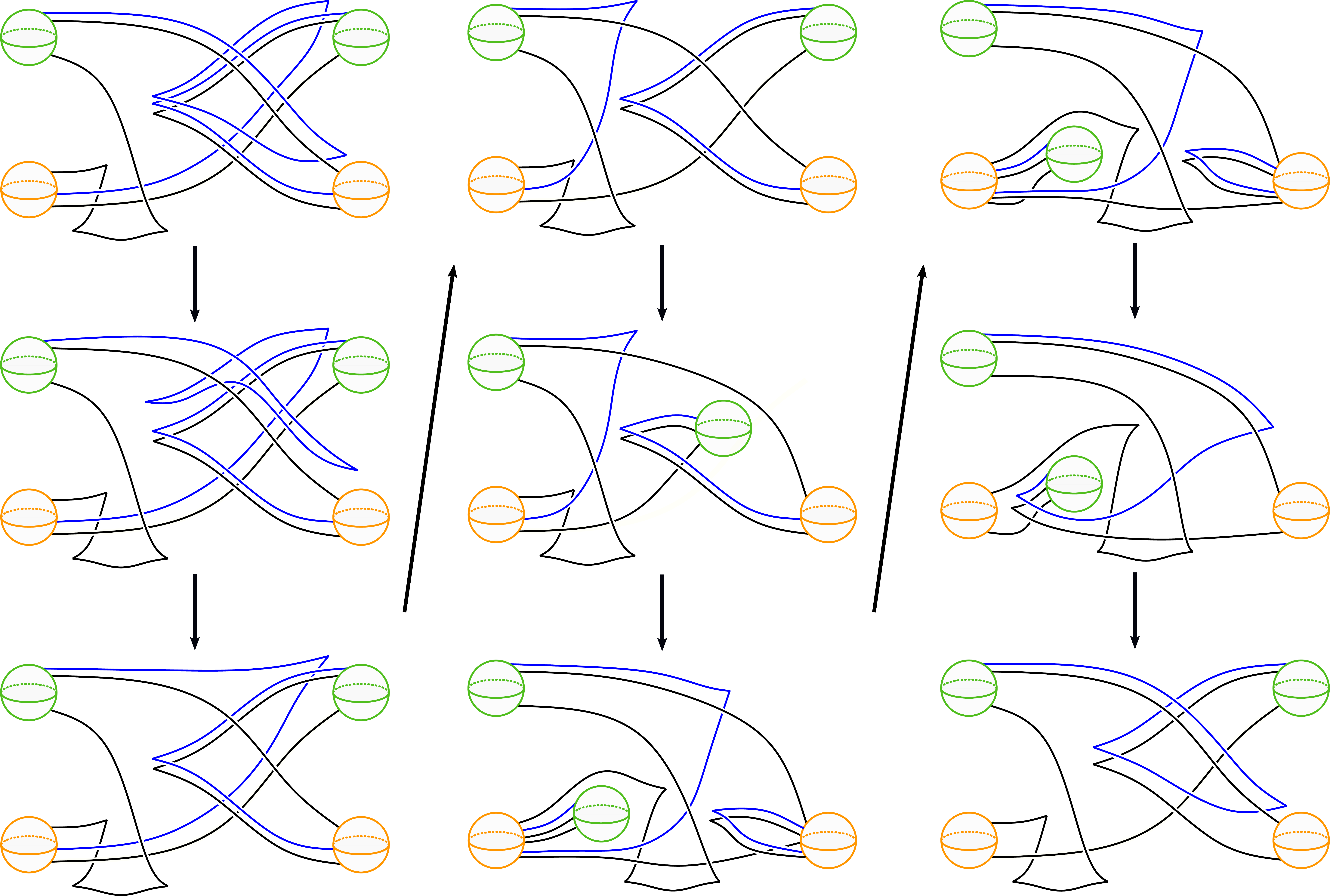}
		\caption{In columns from left to right, a series of Reidemeister and Gompf moves and $1$-handle slides that take the Legendrian lift of a $(1,-1)$ curve to the Legendrian lift of a $(1,0)$ curve.  $(1)$ Reidemeister III, $(2)$ Reidemeister II and I, $(3)$ Reidemeister II, $(4)$ Reidemeister II, $(5)$ slide green $1$-handle over orange $1$-handle, $(6)$ Gompf move $5$, $(7)$ Gompf move $4$, and $(8)$ Reidemeister II.}
		\label{onenodeslide}
	\end{center}
\end{figure}
	
\begin{figure}
\begin{overpic}[scale=.25]{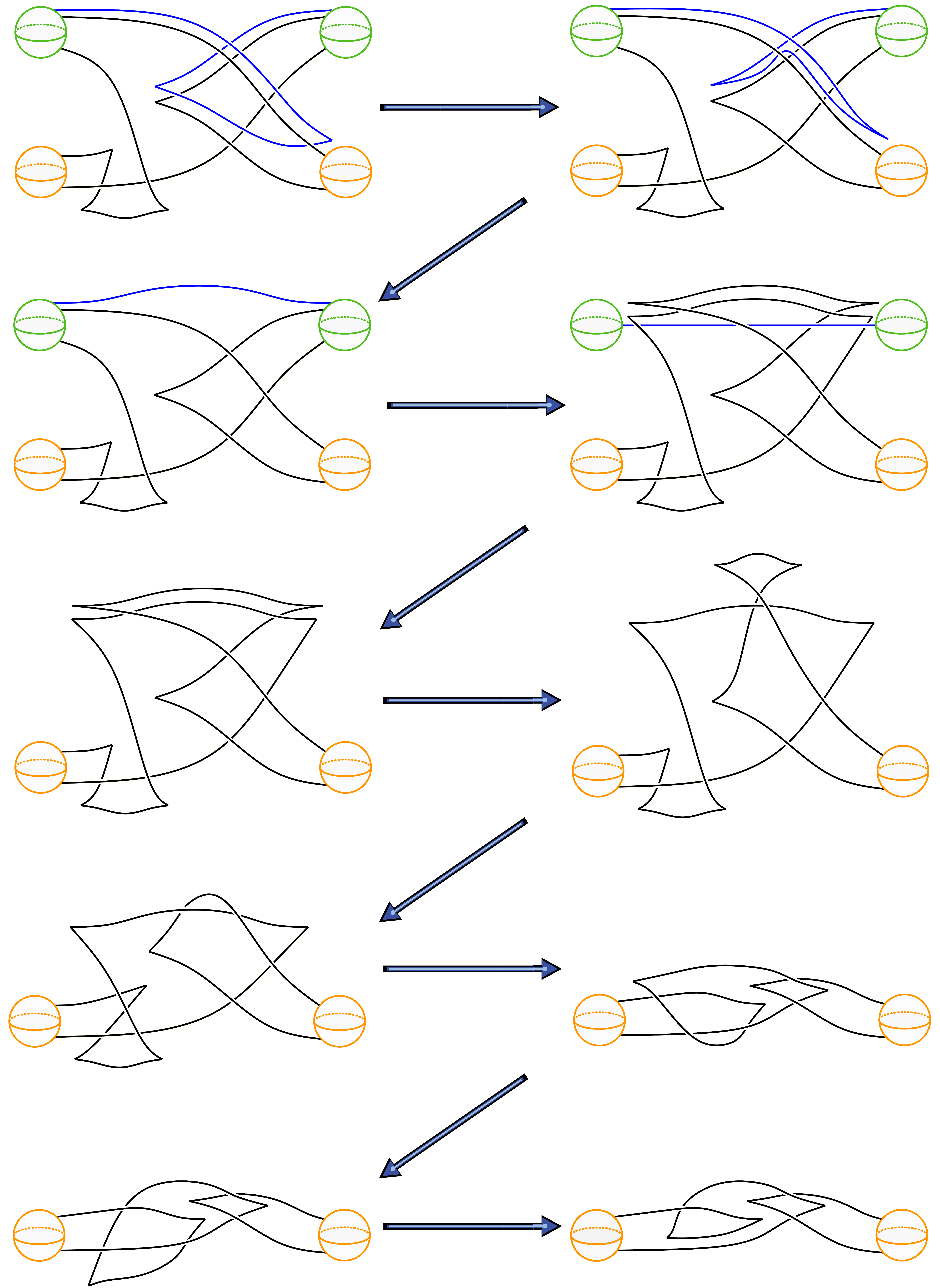}
\put(35, 93){$(1)$}
\put(35, 83){$(2)$}
\put(35, 70){$(3)$}
\put(35, 57){$(4)$}
\put(35, 47){$(5)$}
\put(35, 35){$(6)$}
\put(35, 26){$(7)$}
\put(35, 15){$(8)$}
\put(35, 6){$(9)$}
\end{overpic}
\caption{A series of Reidemeister moves and handle slides simplifying the Weinstein handle diagram of the complement of the toric divisor smoothed in one node. $(1)$ Reidemeister III, $(2)$ Reidemeister II and I, $(3)$ $2$-handle slide, $(4)$ Handle cancellation, $(5)$ Reidemeister III, $(6)$ Reidemeister I and II, $(7)$ Reidemeister III, I and II, $(8)$ Gompf move 6, $(9)$ Reidemeister II.}
\label{onenodeseries}
\end{figure} 

	Although our initial explicit example asked us to attach a handle along a curve of slope $(1,-1)$, we can see that in fact the resulting Weinstein manifold is equivalent to using the $(1,0)$ slope. Figure~\ref{onenodeslide} shows the series of $1$ handle slides, Reidemeister and Gompf moves that take the Legendrian lift of the resulting diagram obtained by attaching along a $(1,-1)$ curve to the diagram corresponding to attaching along a $(1,0)$ curve. \smallskip

\begin{figure}[!ht]
	\begin{center}
		\includegraphics[width=8 cm]{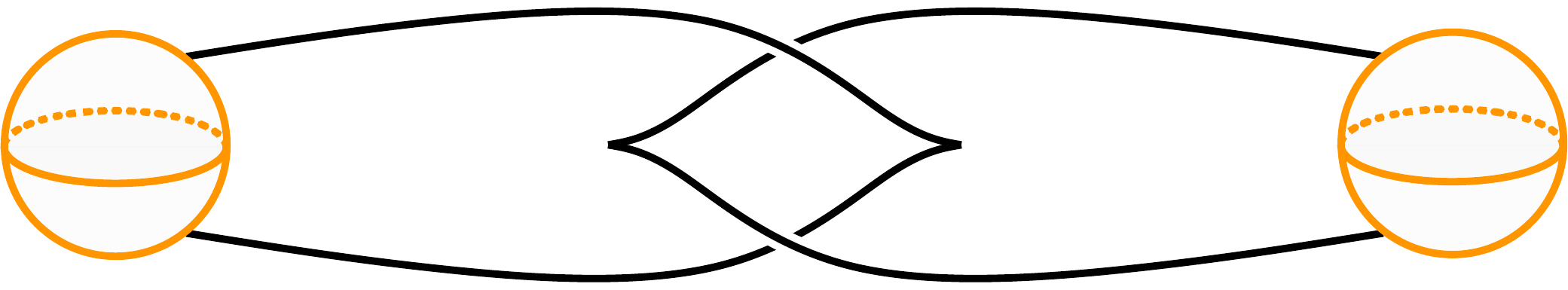}
		\caption{The diagram of the complement of the toric divisor smoothed in one node after simplifications, in particular after applying a single Reidemeister I to the last diagram in Figure \ref{onenodeseries}.}
		\label{onenodefinal}
	\end{center}
\end{figure}

\begin{prop}
	For any toric $4$-manifold, let $\Sigma$ be the result of smoothing the toric divisor at a single node. Then the complement of a small regular neighborhood of $\Sigma$ is a Weinstein domain which is Weinstein homotopic to the domain represented by the handle diagram of Figure~\ref{onenode}.
\end{prop} 	
\begin{proof}
First note that for a single node, the centered hypothesis is automatically satisfied. By Theorem~\ref{thm:toricWein} such a complement is a Weinstein domain obtained by attaching a single 2-handle to $D^*T^2$ along the Legendrian curve given by a co-oriented co-normal lift of a curve of slope $(a,b)$ in the torus. We verify that the resulting Weinstein domain does not depend on the choice of the slope up to Weinstein homotopy. We can see that $T^*T^2 \cup \Lambda_{(a,b)}$ is Weinstein homotopic to $T^*T^2 \cup \Lambda_{(1,0)}$, by performing 1-handle slides on $T^*T^2 \cup \Lambda_{(a,b)}$ similar to Figure \ref{onenodeslide}, to take a $\Lambda_{(a,b)}$ curve to $\Lambda_{(a,b \pm a)}$ or to $\Lambda_{(a\pm b,b)}$. Using the Euclidean algorithm, with an appropriate choice of 1-handle slides, one can start with $T^*T^2 \cup \Lambda_{(a,b)}$, for any pair $a,b\in \Z$ that are relatively prime, and end with $T^*T^2\cup \Lambda_{(1,0)}$. That the Weinstein domains are symplectomorphic can also be proved using toric arguments (see \cite[Proposition 5.5]{ACGMMSW}).
\end{proof}

The diagram where we attach along the $(1,0)$ curve is preferable to the one where we attach along slope $(1,-1)$ (or a more complicated $(a,b)$ curve) because it is easier to simplify. Starting with Figure \ref{onenode}, we can perform Reidemeister moves, Gompf moves, handle cancellations and handle slides. We choose such a simplifying sequence in Figure \ref{onenodeseries} to obtain the diagram illustrated in Figure \ref{onenodefinal}.

\vskip.1cm


\section{A more complicated example: smoothing a toric divisor in $\C\P^2 \# 3\cptwobar$}

Consider $\cptwo$ blown up three times with the same size of the blow ups. The corresponding Delzant polytope is a hexagon illustrated in Figure \ref{hextor}. Observe that the sizes of the blow-ups have been chosen precisely to make the polytope centered with respect to all of its vertices. The inward normals of the six corners are: $(1,0),( 0,-1),(-1,-1),(-1,0),(0,1), (1,1)$. By Theorem~\ref{thm:toricWein}, if we smooth all six singularities that map under the moment map to the vertices of the hexagon polytope, the complement of the smoothed divisor is $T^*T^2$ with $2$-handles attached along 
 $$\Lambda_{(1,0)},\Lambda_{(0,-1)}, \Lambda_{(-1,-1)},\Lambda_{(-1,0)}, \Lambda_{(0,1)},\Lambda_{(1,1)}$$
as shown in Figure \ref{hex}. 

\begin{figure}[!ht]
  \centering
  \includegraphics[width=6cm]{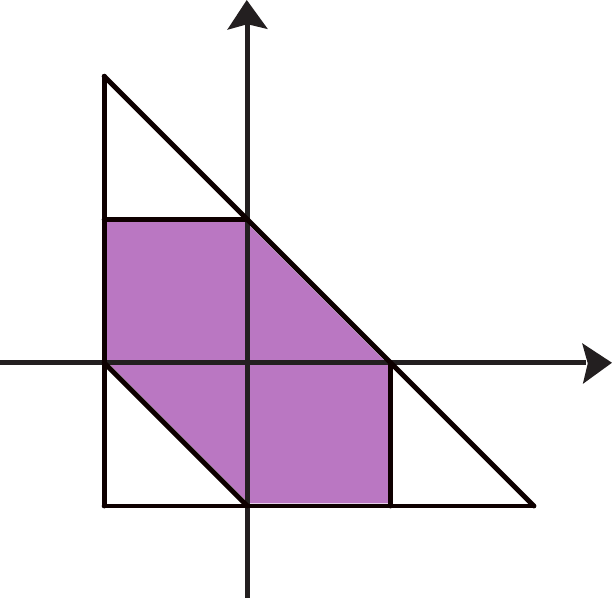}
  \caption{The Delzant polytope of  monotone $\cptwo \# 3\cptwobar$.}
  \label{hextor}
\end{figure}

\begin{figure}[ht]
	\begin{center}
		\includegraphics[width=10cm]{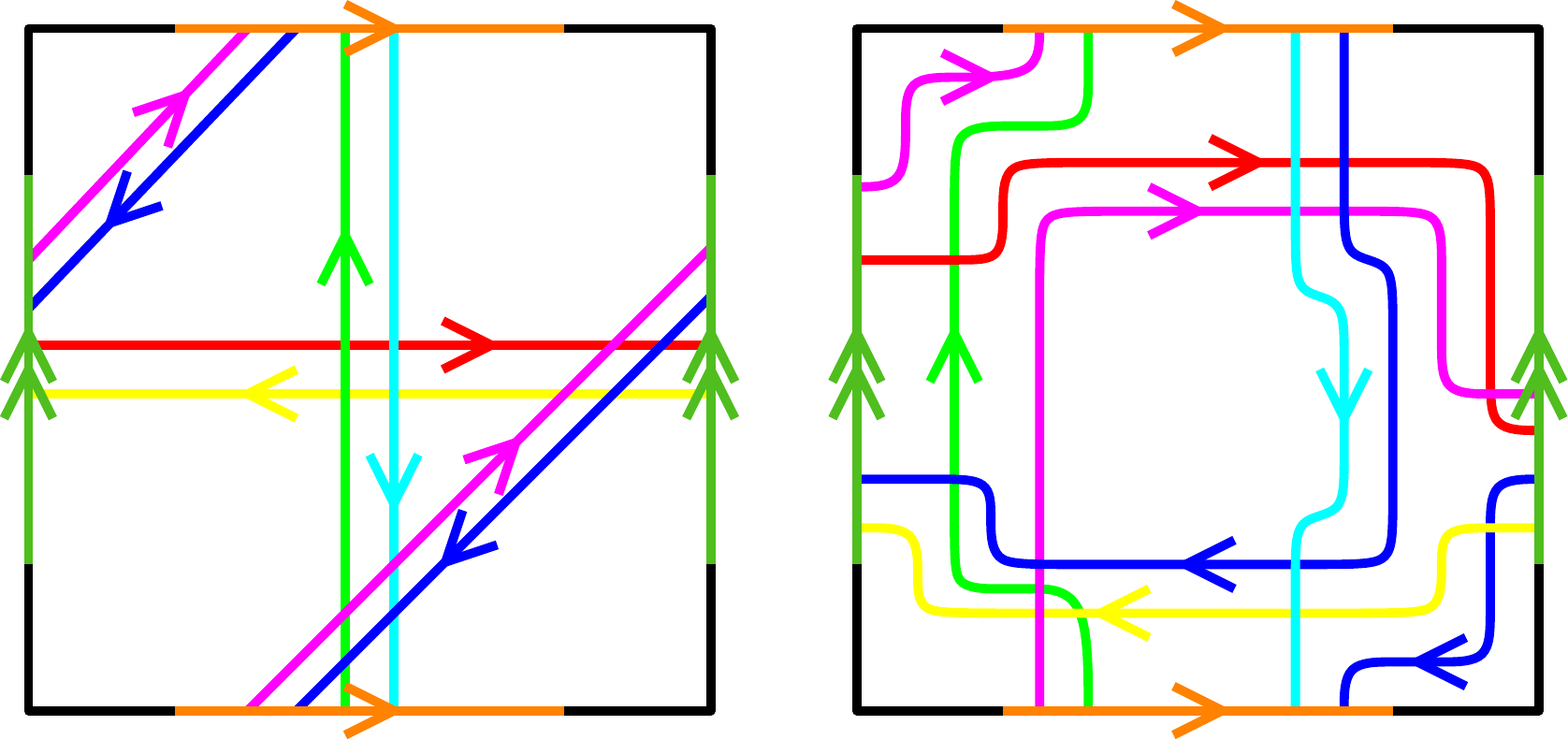}
		\caption{The curves $(1,0),( 0,-1),(-1,-1),(-1,0),(0,1), (1,1)$ on $T^2$ before and after isotoping them to agree with concentric circles co-oriented inward (except where they pass through the $1$-handles).}
		\label{hex}
	\end{center}
\end{figure}

Figure \ref{hex} shows the curves $\Gamma=(1,0),( 0,-1),(-1,-1),(-1,0),(0,1), (1,1)$ on $T^2$. As in the previous section, we isotope the curves giving a Legendrian isotopy of their co-normal lifts, in order to have them run parallel to the boundary of the $0$-handle in $T^2$ except possibly at $1$-handles. This allows us to identify them in $J^1(S^1)$. We isotope each curve so that at the end, the co-orientation points inward. Since there are multiple curves, we choose an isotopy which minimizes the number of crossings. 

\begin{figure}[ht]
	\begin{center}
		\includegraphics[width=10cm]{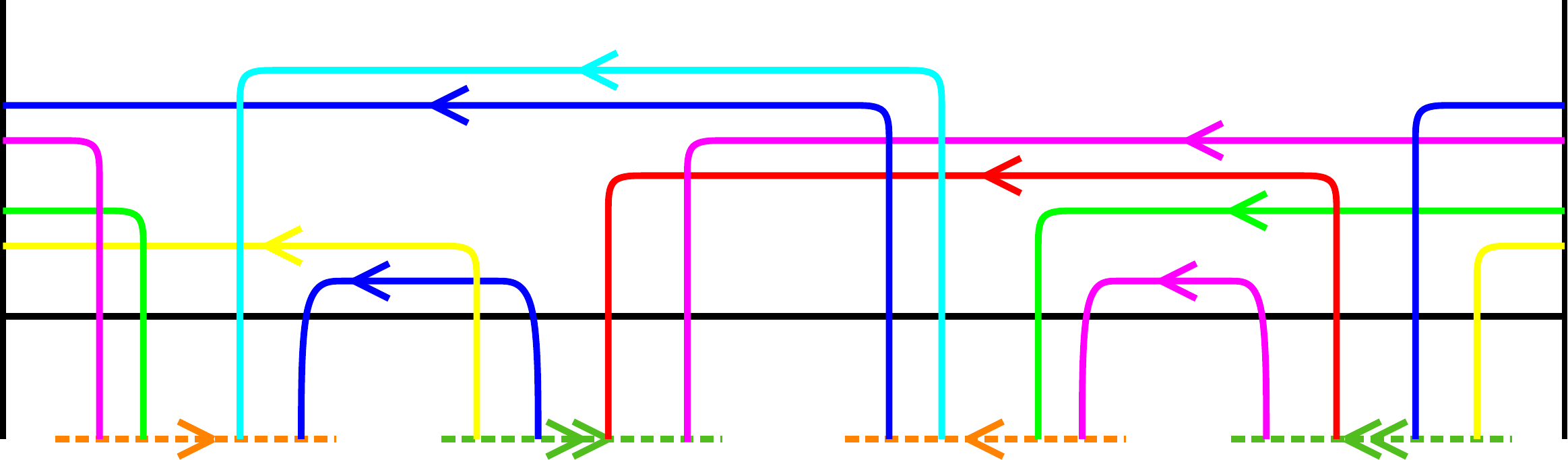}
		\caption{The curves in $J^1(S^1)$ corresponding to the curves $(1,0),( 0,-1),(-1,-1),(-1,0),(0,1), (1,1)$.}
		\label{hexjet}
	\end{center}
\end{figure}

Once these curves agree (except where they enter the $1$-handles) with parallel copies of suitable concentric circles which are positive Reeb push-offs of boundary of the $0$-handle of the $T^2$ $0$-section, we identify a neighborhood of boundary of the square (containing all our isotoped curves) with a neighborhood of the zero-section of $J^1(S^1)$. This identifies all our curves in $T^2$ with curves in $J^1(S^1)$ as in Figure \ref{hexjet}. We then satellite the image of the curves in $J^1(S^1)$ onto the image of $S^1$ in the Gompf diagram of $T^*T^2$, using the conventions described in the previous section and illustrated in Figures~\ref{fig:gradient} and~\ref{curvecorner} to maintain the relative positions of the curves. The result is Figure~\ref{hexleg}. 

\begin{figure}[ht]
	\begin{center}
		\includegraphics[width=8cm]{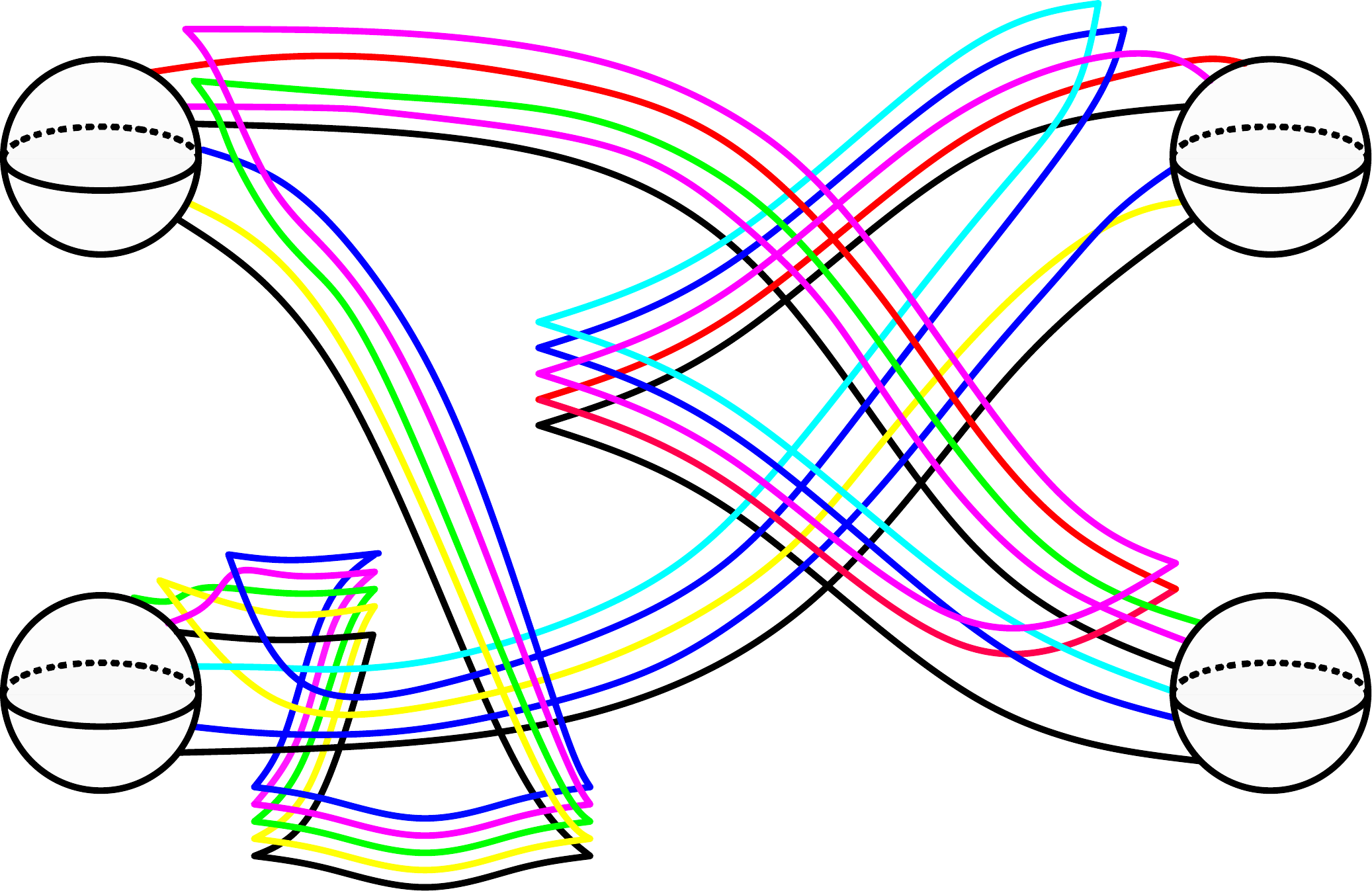}
		\caption{A Legendrian handle diagram obtained by simplifying the Legendrian handlbody diagram of the complement of the toric divisor of $\cptwo\#3\overline{\cptwo}$ smoothed in six nodes.}
		\label{hexleg}
	\end{center}
\end{figure}

\begin{figure}[ht]
	\begin{center}
		\includegraphics[width=8 cm]{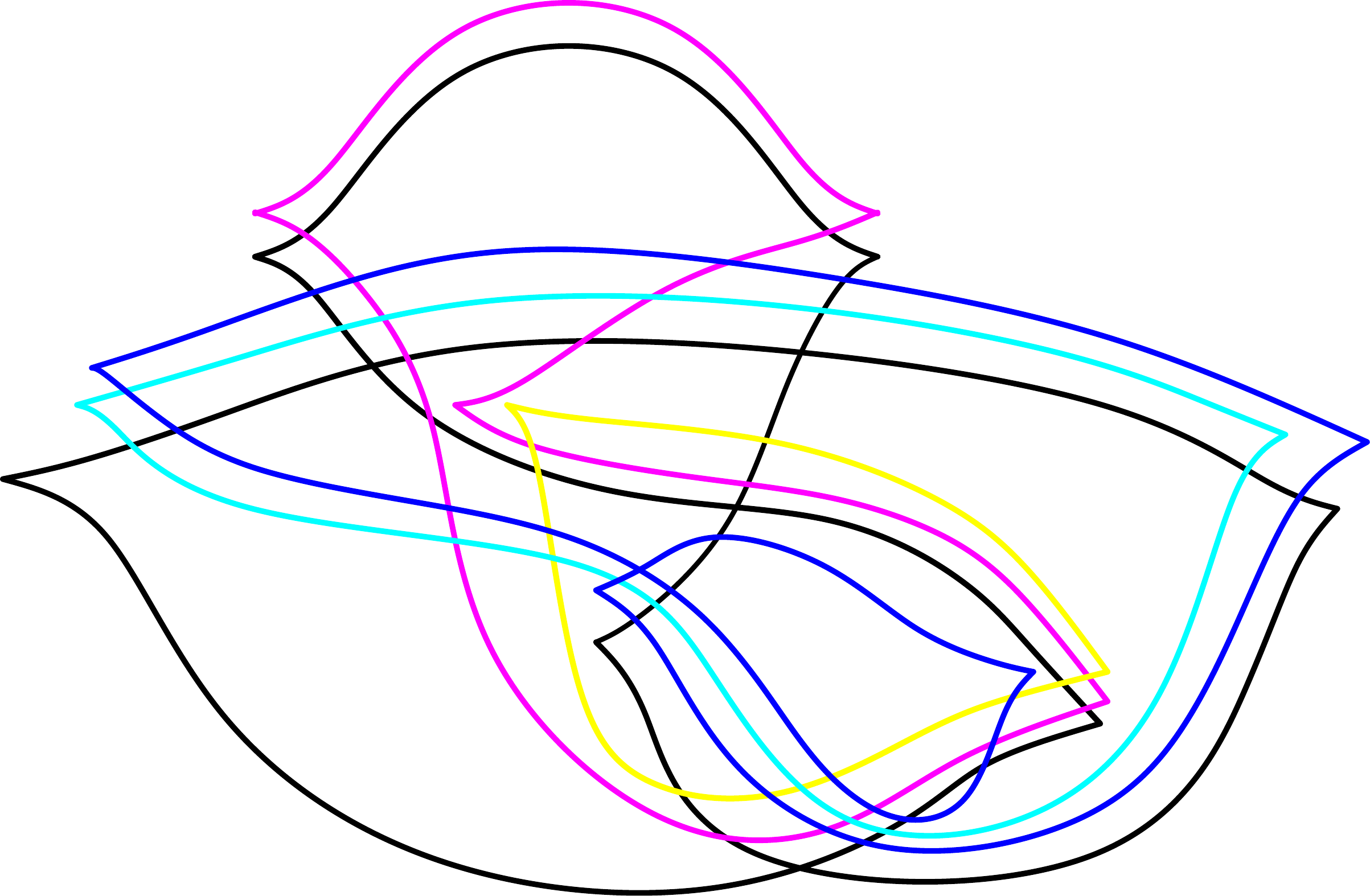}
		\caption{A Legendrian link obtained by simplifying the Legendrian handle diagram of the complement of the toric divisor of $\cptwo\#3\overline{\cptwo}$ smoothed in six nodes.}
		\label{hexfinal}
	\end{center}
\end{figure}

Performing a sequence of handle slides and Legendrian isotopy yields a simpler Weinstein handle diagram shown in Figure \ref{hexfinal}. The Weinstein handle diagram allows us to easily compute the homology of $\cptwo \# 3 \overline{\cptwo}\setminus\nu(\widetilde{\Sigma_1})$, where $\widetilde{\Sigma_1}$ is the smoothed toric divisor and $\nu(\widetilde{\Sigma_1})$ is the neighborhood of $\widetilde{\Sigma_1}$. The handle structure determines the cellular/Morse homology chain complex, which allows us to determine the homology groups. A presentation for the fundamental group can also be computed from the handle decomposition where the $1$-handles correspond to generators and $2$-handles provide relations.

In this example we obtain:
\begin{align*}
\pi_1(\cptwo \# 3 \overline{\cptwo}\setminus\nu(\widetilde{\Sigma_1}); \Z) &= 0, \\
H_0(\cptwo \# 3 \overline{\cptwo}\setminus\nu(\widetilde{\Sigma_1}); \Z) &= \Z, \\
H_1(\cptwo \# 3 \overline{\cptwo}\setminus\nu(\widetilde{\Sigma_1}); \Z) &= 0, \\
H_2(\cptwo \# 3 \overline{\cptwo}\setminus\nu(\widetilde{\Sigma_1}); \Z) &= \Z^{\oplus 5} \text{ with intersection form } 
  \left[ {\begin{array}{ccccc}
   0 & -2 & 1 & 0 & 0\\
   -2 & -4 & -1 & 1 & 1\\
   1 & -1 & -2 & 1 & 0\\
   0 & 1 & 1 & -2 & -1\\
   0 & 1 & 0 & -1 & -3\\
  \end{array} } \right]
,\\
H_i(\cptwo \# 3 \overline{\cptwo}\setminus\nu(\widetilde{\Sigma_1}); \Z) &=0 \text{ for $i=3,4$.}
\end{align*}


Using the same toric diagram, instead of smoothing all six nodes, we can choose to smooth only a subset of them. The symplectic topology of the corresponding complement will differ (with fewer handles attached). We consider here an example yielding a Weinstein manifold of interest.
 Let us smooth three alternating nodes in $\cptwo \# 3 \overline{\cptwo}$. The complement of the smoothed divisor given by  $T^*T^2$ with $2$-handles attached along $\Lambda_{(1,0)}, \Lambda_{(-1,-1)}, \Lambda_{(0,1)}$ is pictured in Figure \ref{althexstart}. 
 Performing a sequence of handle slides and Legendrian isotopy on Figure \ref{althexstart} yields a simpler Weinstein handle diagram shown in Figure~\ref{althexfinal} consisting of a Legendrian (2,4) torus knot. 

\begin{figure}[ht]
\begin{center}
\includegraphics[width=8 cm]{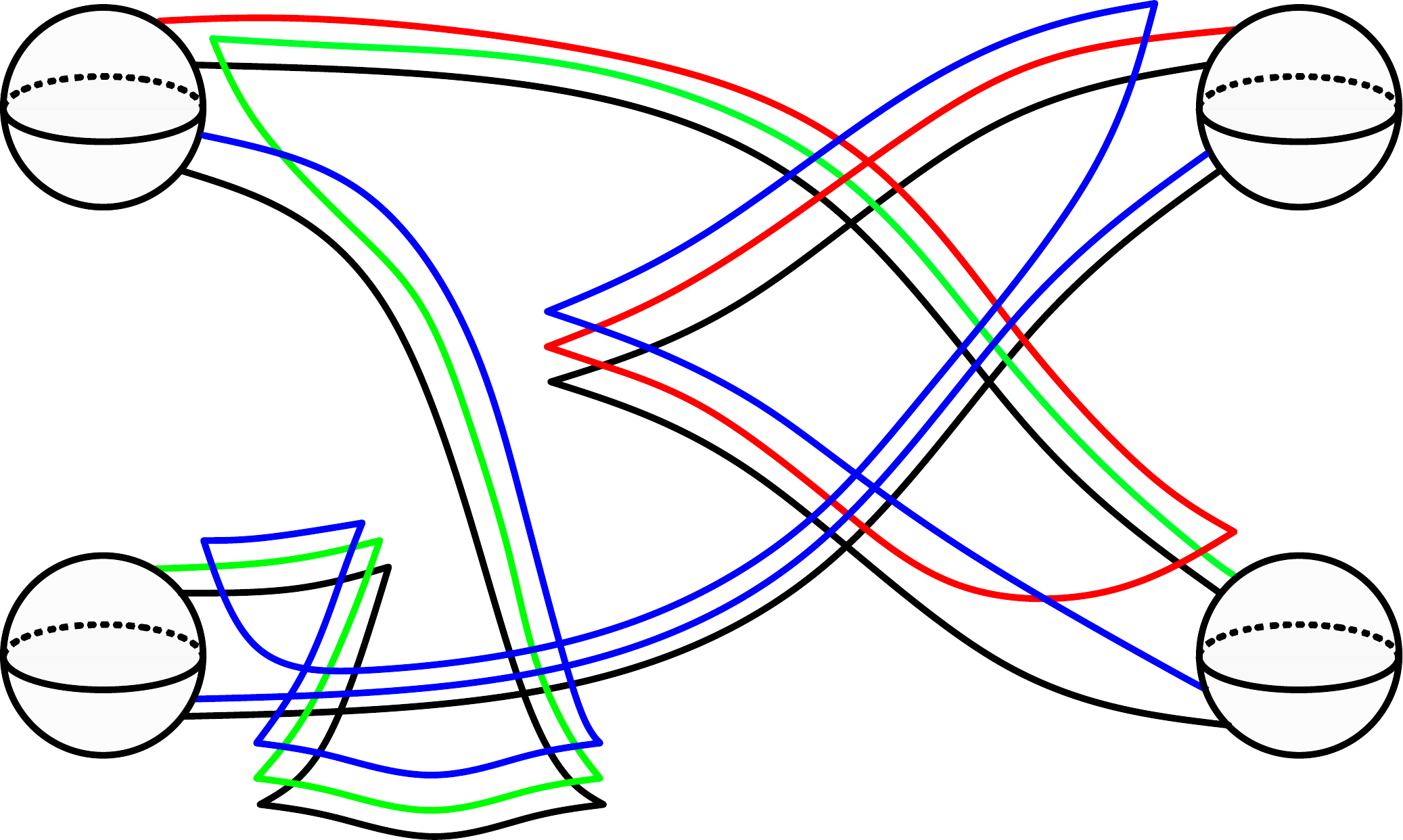}
\caption{A Legendrian handle diagram obtained by simplifying the Legendrian handlbody diagram of the complement of the toric divisor of $\cptwo\#3\overline{\cptwo}$ smoothed in three nodes.}
\label{althexstart}
\end{center}
\end{figure}


\begin{figure}[ht]
\begin{center}
\includegraphics[width=5 cm]{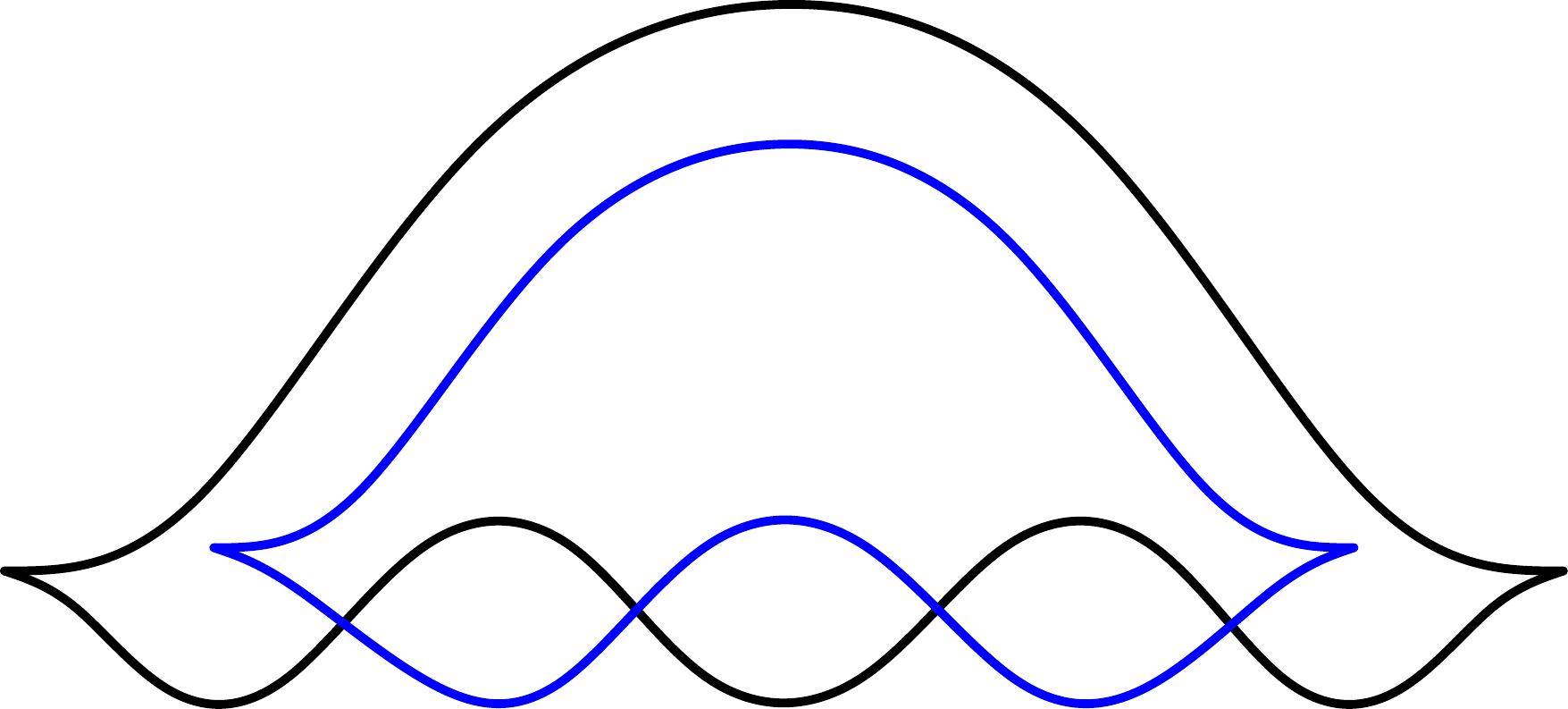}
\caption{A Legendrian (2,4)-torus link obtained from Figure \ref{althexstart} by a sequence of handleslides and Legendrian isotopy.}
\label{althexfinal}
\end{center}
\end{figure}

In this picture we can see two Lagrangian spheres in distinct homology classes in the complement of this smoothed divisor. Again, the homology of $\cptwo \# 3 \overline{\cptwo}\setminus\nu(\widetilde{\Sigma_2})$, where $\widetilde{\Sigma_2}$ is the smoothed toric divisor and $\nu(\widetilde{\Sigma_2})$ is the neighborhood of $\widetilde{\Sigma_2}$ can be computed from the Weinstein handle diagram. In particular,
\begin{align*}
\pi_1(\cptwo \# 3 \overline{\cptwo}\setminus\nu(\widetilde{\Sigma_2}); \Z) &= 0, \\
H_0(\cptwo \# 3 \overline{\cptwo}\setminus\nu(\widetilde{\Sigma_2}); \Z) &= \Z, \\
H_1(\cptwo \# 3 \overline{\cptwo}\setminus\nu(\widetilde{\Sigma_2}); \Z) &= 0, \\
H_2(\cptwo \# 3 \overline{\cptwo}\setminus\nu(\widetilde{\Sigma_2}); \Z) &= \Z\oplus\Z \text{ with intersection form } \left[ {\begin{array}{cc}
   -2 & 2 \\
   2 & -2 \\ 
  \end{array} } \right]
,\\
H_i(\cptwo \# 3 \overline{\cptwo}\setminus\nu(\widetilde{\Sigma_2}); \Z) &=0 \text{ for $i=3,4$.}
\end{align*}


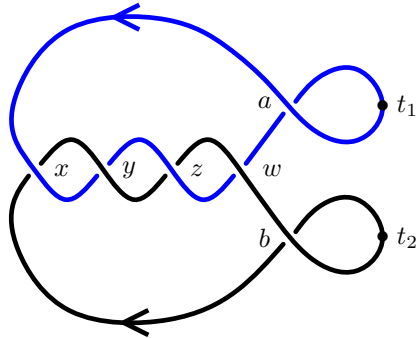
\begin{figure}[h]
	\begin{center}
		\begin{tikzpicture}
			\node[inner sep=0] at (0,0) {\includegraphics[width=5 cm]{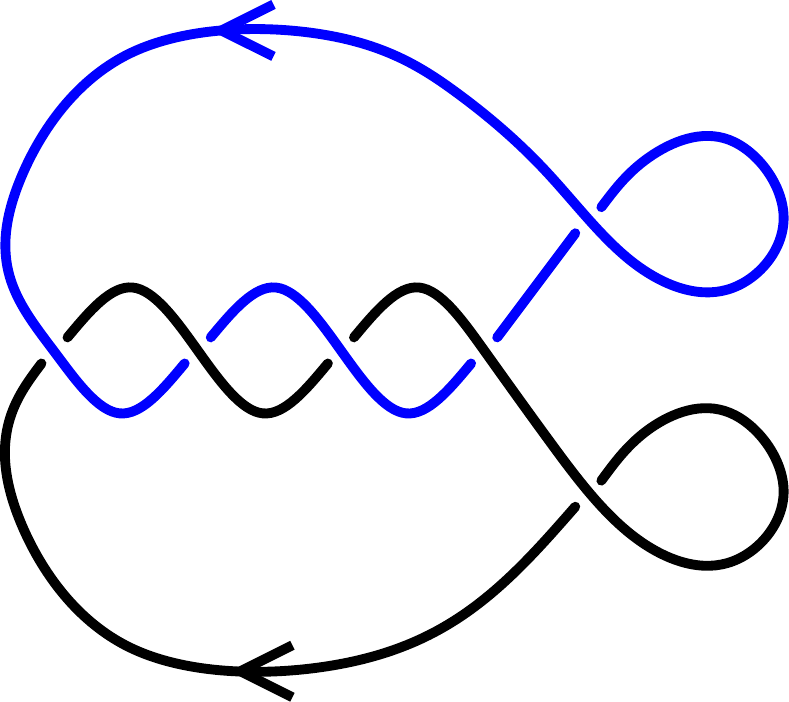}};
			\node at (-1.8,0){$x$};
			\node at (-0.9,0){$y$};
			\node at (0,0){$z$};
			\node at (1,0){$w$};
			\node at (0.9,0.9){$a$};
			\node at (0.9,-0.9){$b$};
			\node at (2.47,0.85){\textbullet};
			\node at (2.47,-0.9){\textbullet};
			\node at (2.8,0.85){$t_1$};
			\node at (2.8,-0.9){$t_2$};
		\end{tikzpicture}
		\caption{A diagram of the complement of the toric divisor of $\cptwo\#3\overline{\cptwo}$ smoothed in three nodes in the Lagrangian projection with labelled Reeb chords and marked points generating the Chekanov-Eliashberg DGA.}
		\label{althexlag}
	\end{center}
\end{figure}

The Weinstein diagram allows us to also compute the DGA of $\Lambda \in S^3$, $\mathcal{A}(\Lambda)$ (see~\cite{NE} for a good survey on how to calculate the Legendrian DGA). In this case is generated by $a,b,x,y,z,w$ and marked points $t_1$ and $t_2$, as shown in Figure \ref{althexlag}. The generators have gradings
$$|a|=|b|=1, |x|=|y|=|z|=|w|=0$$
with differential
\begin{align*}
\partial x &= \partial y=\partial z =\partial w=0\\
\partial a   &= xyzw+xw+xy+zw+1+t_1^{-1}\\
\partial ab  &= wzyx+yx+wz+wx+1+t_2.
\end{align*}

The DGA of $\Lambda$ is an interesting invariant to consider because the Wrapped Fukaya chain complex of $\cptwo \# 3 \overline{\cptwo}\setminus\nu(\widetilde{\Sigma_2})$ is $A_{\infty}$-quasi isomorphic to the $DGA$ $\mathcal{A}(\Lambda)$~\cite{BEE,Ekholm,EkholmLekili}. Note that $\mathcal{A}(\Lambda)$ has no negatively graded Reeb chords and therefore the degree-$0$ Legendrian contact homology of $\mathcal{A}(\Lambda)$, denoted by $LCH_0(\Lambda)$, is finitely generated. Furthermore, it is expected that $LCH_0(\Lambda)$ is Morita equivalent to a commutative ring $R$ such that $\hat{X}=Spec (R)$ where $\hat{X}$ is the mirror of $\cptwo \# 3 \overline{\cptwo}\setminus\nu(\widetilde{\Sigma_2})$. See~\cite{CM, EtguLekili} for other examples of such a phenomenon.


\bibliography{references}

\bibliographystyle{amsalpha}

\end{document}